\documentclass[fleqn,3p,12pt,times]{elsarticle}
\usepackage[utf8]{inputenc}
\usepackage{amssymb,latexsym}
\usepackage{amsfonts}
\usepackage{amsmath}
\usepackage{enumerate}
\usepackage{bbm}
\usepackage{xcolor}
\usepackage{booktabs}
\usepackage{float}
\usepackage{amssymb}
\usepackage{amsthm}
\usepackage{comment}
\usepackage{url}
\usepackage{enumitem}
\usepackage{natbib}

\makeatletter
\newcommand{\eqnum}{\leavevmode\hfill\refstepcounter{equation}\textup{\tagform@{\theequation}}}
\makeatother

\newtheorem{ex}{Example}
\newtheorem{thm}{Theorem}
\newtheorem{lem}{Lemma}

\newtheorem{con}{Conjecture}

\newcommand{\T}{\mathbb{P}}
\newcommand{\E}{\mathbb{E}}
\newcommand{\vphi}{\varphi}
\newcommand\dist{\buildrel d \over =}

\journal{arXiv}

\newcommand\al{\alpha}

\newcommand\la{\lambda}
\newcommand\ka{\kappa}
\def\leq{\leqslant}
\def\geq{\geqslant}
\newcommand\N{{\mathbb N}}

\newcommand{\f}{\varphi}

\begin{document}
\begin{frontmatter}

\title{Distribution of shifted discrete random walk generated by distinct random variables and applications in ruin theory
}

\author{Simonas Gervė\fnref{myfootnote1}}
\author{Andrius Grigutis\fnref{myfootnote2}}
\address{Institute of Mathematics, Vilnius University, Naugarduko 24, LT-03225 Vilnius}

\fntext[myfootnote1]{Email: simonas.gerve@gmail.com}
\fntext[myfootnote2]{Email: andrius.grigutis@mif.vu.lt}

\begin{abstract}
In this paper, we set up the distribution function 
$$
\vphi(u)=\T\left(\sup_{n\geq1}\sum_{i=1}^{n}\left(X_i-\ka\right)<u\right),
$$
and the generating function of $\vphi(u+1)$, where $u\in\N_0$, $\ka\in\mathbb{N}$, the random walk $\left\{\sum_{i=1}^{n}X_i, n\in\mathbb{N}\right\},$ consists of $N\in\N$ periodically occurring distributions, and the integer-valued and non-negative random variables $X_1,\,X_2,\,\ldots$ are independent. This research generalizes two recent works where $\{\ka=1,\,N\in\N\}$ and $\{\ka\in\N,\,N=1\}$ were considered respectively. 
The provided sequence of sums $\left\{\sum_{i=1}^{n}\left(X_i-\ka\right),\,n\in\N\right\}$ generates so-called multi-seasonal discrete-time risk model with arbitrary natural premium and its known distribution enables to calculate the ultimate time ruin probability $1-\vphi(u)$ or survival probability $\vphi(u)$. Verifying obtained theoretical statements we demonstrate several computational examples for survival probability $\vphi(u)$ and its generating function when $\{\ka=2,\,N=2\}$, $\{\ka=3,\,N=2\}$, $\{\ka=5,\,N=10\}$ and $X_i$ admits Poisson and some other distributions. We also conjecture the non-singularity of certain matrices.

\end{abstract}

\begin{keyword}
multi-seasonal discrete-time risk model, survival probability, random walk, initial values, generating functions, ruin theory.
\MSC[2020] 91G05, 60G50, 60J80.
\end{keyword}

\end{frontmatter}

\section{Introduction}

Many probabilistic models, estimating likelihoods of certain events, are based on the sequence of sums $\left\{\sum_{i=1}^{n}X_i,\,n\in\mathbb{N}\right\}$, where $X_i$ are some random variables. This sequence is called {\it the random walk} (r. w.). Random walks are usually visualized as branching trees or certain paths in plane or space; their occurrence spreads from pure mathematics to many applied sciences. For instance, one may refer to Case–Shiller home pricing index \cite{Case-Shiller} or even more generally to the random walk hypothesis \cite{RW_H}. From a pure mathematics standpoint, one may mention the random matrix theory, see for example \cite{ARGUIN2022174}, \cite{edelman_rao_2005}, \cite{martinsson_tropp_2020} and other related works. Perhaps the closest context where the need to know the distribution of $\sup_{n\geq1}\sum_{i=1}^{n}\left(X_i-\ka\right)$ arises is insurance mathematics. In ruin theory one may assume that insurer's wealth $W(n)$ in discrete time moments $n\in\mathbb{N}$ consists of incoming cash premiums and outgoing payoffs (claims), and $W(n)$ admits the following representation:
\begin{align}\label{model}
W(n)=u+\kappa n-\sum_{i=1}^{n}X_i=u-\sum_{i=1}^{n}\left(X_i-\ka\right),
\end{align}
where $u\in\mathbb{N}_0:=\mathbb{N}\cup\{0\}$ is interpreted as initial surplus $W(0):=u$, $\kappa\in\mathbb{N}$ denotes the arrival rate of premiums paid by customers and the subtracted sum of random variables represents claims. Here we assume that random variables $X_i$ are independent, non-negative and integer-valued but not necessarily identically distributed. More precisely, $X_i\dist X_{i+N}$ for all $i\in\mathbb{N}$ and some fixed $N\in\mathbb{N}$. The model \eqref{model} can be visualized as a ''race'' between deterministic line $u+\ka n$ and the sum of random variables $\sum_{i=1}^{n}X_i$ when $n$ varies, see Figure \ref{fig:race}\footnote{Implemented using the RandomChoice function in Wolfram Mathematica \cite{Mathematica}}.

\begin{figure}[H]
\begin{center}
\includegraphics[scale=1.1]{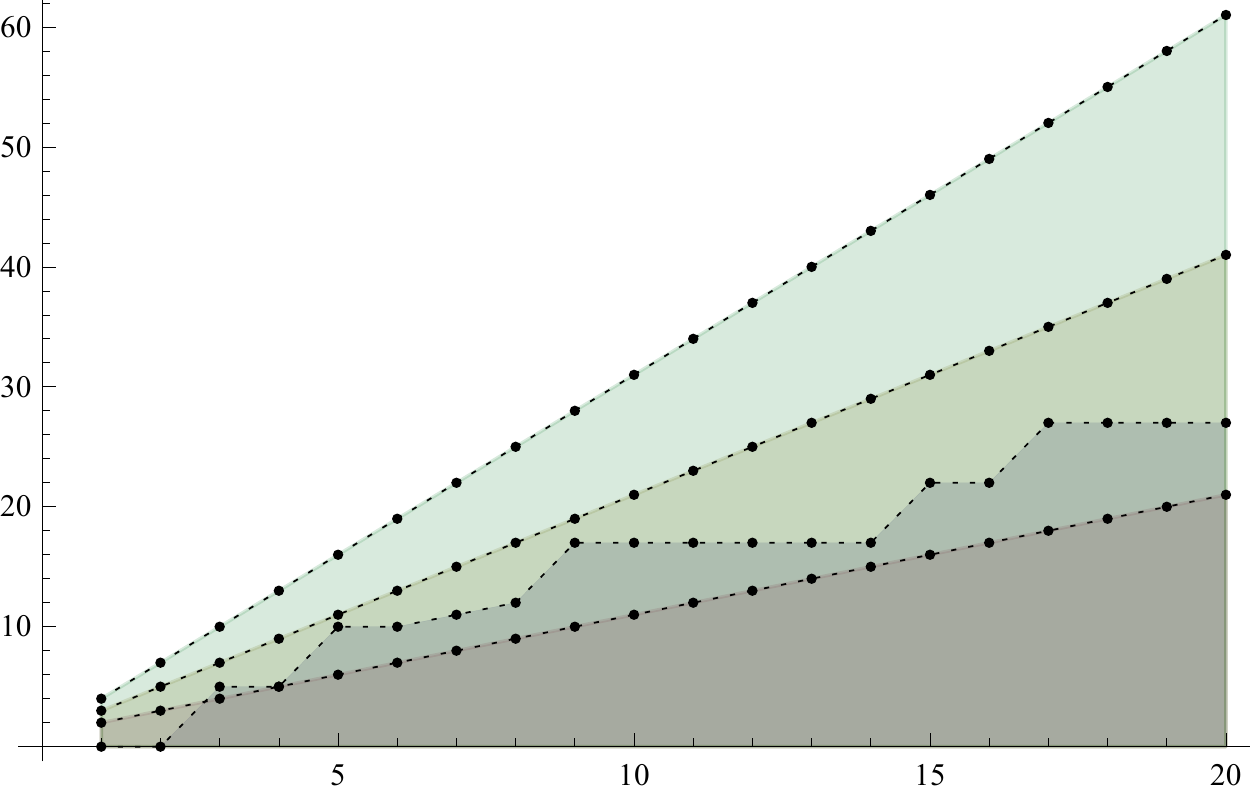}
\end{center}
\caption{Lines $1+n$, $1+2n$, $1+3n$, and random walk $\sum_{i=1}^{n}X_i\mathbbm{1}_{\{i\mod 2=1\}}+\sum_{i=1}^{n}Y_i\mathbbm{1}_{\{i\mod 2=0\}}$, where $\mathbb{P}(X_i=0)=0.3$, $\mathbb{P}(X_i=1)=0.1$, $\mathbb{P}(X_i=5)=0.6$ and $\mathbb{P}(Y_i=0)=0.8$, $\mathbb{P}(Y_i=1)=0.1$, $\mathbb{P}(Y_i=10)=0.1$, and $n$ varies from $1$ to $20$.}\label{fig:race}
\end{figure}

We say that the fixed natural number $N$ is the number of {\it periods} or {\it seasons} and call the model \eqref{model} {\it N-seasonal discrete-time risk model}.
The model \eqref{model} with $N=1$ is a discrete version of the more general continuous time Cramér – Lundberg model (also known as classical risk process)
\begin{align}\label{clas_mod}
\tilde{W}(t)=x+ct-\sum_{i=1}^{P_t}{\xi_i},\,t\geqslant0,
\end{align}
where, analogously as in model \eqref{model}, $x\geqslant0$ represents initial surplus, $c>0$ premium, $\xi_i$ independent and identically distributed non-negative random variables, and $P_t$ is the counting Poisson process with intensity $\lambda>0$. The model \eqref{clas_mod} can be further extended, cf. E. Spare Andersen's model \cite{Andersen}. 

Being curious whether initial surplus and collected premiums can always cover incurred claims, for the N-seasonal discrete-time risk model \eqref{model} we define {\it the finite time survival probability}
\begin{align}\label{fin_time_p}
\vphi(u,\,T)=\T\left(\bigcap_{n=1}^{T}\left\{W(n)>0\right\}\right)=\T\left(\sup_{1\leq n\leq T}\sum_{i=1}^{n}\left(X_i-\ka\right)<u\right),
\end{align}
where $T$ is some fixed natural number, and {\it the ultimate time survival probability}
\begin{align}\label{ult_time_p}
\vphi(u)=\T\left(\bigcap_{n=1}^{\infty}\left\{W(n)>0\right\}\right)=\T\left(\sup_{n\geq1}\sum_{i=1}^{n}\left(X_i-\ka\right)<u\right).
\end{align}
Calculation of finite time survival (or ruin, $\psi(u,\,T):=1-\vphi(u,\,T))$ probability \eqref{fin_time_p} is far easier than the calculation of ultimate time survival (or ruin, $\psi(u)=1-\vphi(u)$) probability \eqref{ult_time_p}, see Theorem \ref{thm:fin_t} below for $\vphi(u,\,T)$ expressions. Difficulties calculating $\vphi(u)$ arise due to $\vphi(\ka N),\, \vphi(\ka N+1),\,\ldots$ being expressible via $\vphi(0),\,\ldots,\,\vphi(\ka N-1)$ which are initially unknowns, see the formula \eqref{eq:rec} in the next section. Therefore, the essence of the problem we solve is nothing but finding these initial values $\vphi(0),\,\ldots,\,\vphi(\ka N-1)$. In this work, we demonstrate that the required values of $\vphi(0),\,\ldots,\,\vphi(\ka N-1)$ satisfy a certain system of linear equations (see the system \eqref{main_syst}), coefficients of which are based on certain polynomials and the roots of $s^{\ka N}=G_{S_N}(s)$, where $s\in\mathbb{C}$ and $G_{S_N}(s)$ is the probability-generating function of $S_N=X_1+\ldots+X_N$. 

Let us briefly overview the history and some classical works on the subject and mention a few recent papers. The foundation of ruin theory dates back to 1903 when Swedish actuary Filip Lundberg published his work \cite{Lundberg}, which was republished in 1930 also by Swedish mathematician Harald Cramér, while the random walk formulation as such was first introduced by English mathematician and biostatistician Karl Pearson \cite{Pearson}. Cramér-Lundberg riks model \eqref{clas_mod} was extended by Danish mathematician Erik Sparre Andersen by allowing claim inter-arrival times to have arbitrary distribution \cite{Andersen}. The next famous works were published in the late eighties by Hans U. Gerber and Elias S. W. Shiu, see \cite{Gerber}, \cite{Gerber1}, \cite{Shiu}, \cite{Shiu1}. Equally, in the second half of the twentieth century, there were many sound studies regarding the random walk by such authors as William Feller, Frank L. Spitzer, David G. Kendal, Félix Pollaczek and others, see \cite{Feller}, \cite{Spitzer}, \cite{Spitzer1}, \cite{Kendal}, \cite{Pollaczek} and related works. Scrolling across the timeline in recent decades, one may reference a notable survey \cite{Li}. Various assumptions on random walk's structure in models \eqref{model} or \eqref{clas_mod} make recent literature voluminous. With that said, let us mention \cite{dembo}, \cite{rvz}, \cite{asmussen},\cite{Dickson}, \cite{Scandin}, \cite{Picard}, \cite{Li1}, \cite{Lef}, \cite{Yang}, \cite{Cang}.

\section{Recursive nature of ultimate time survival probability, basic notations and the net profit condition}

This section starts with deriving the basic recurrent relation for the ultimate time survival probability. The definition \eqref{ult_time_p}, the law of total probability and rearrangements imply
\begin{align}\label{eq:rec}\nonumber
\varphi(u)&=\mathbb{P}\left(\bigcap_{n=1}^{\infty}\left\{W_u(n)>0\right\}\right)\\
&=\mathbb{P}\left(\bigcap_{n=1}^{N}\left\{u+\kappa n - \sum_{i=1}^n X_i>0\right\} 
\cap
\bigcap_{n=N+1}^{\infty}\left\{u+\kappa n - \sum_{i=1}^n X_i>0\right\}\right)\nonumber\\
&=\mathbb{P}\left(\bigcap_{n=1}^{N}\left\{\sum_{i=1}^nX_i\leqslant u+\kappa n-1\right\}\cap
\bigcap_{n=N+1}^{\infty}\left\{u+\kappa N-\sum_{i=1}^NX_i+\kappa(n- N)-\sum_{i=N+1}^nX_i>0\right\}\right)\nonumber\\
&=\sum_{\substack{i_1\leqslant u+\kappa-1\\i_1+i_2\leqslant u+2\kappa-1\\
\vdots \vspace{1mm} \\
i_1+i_2+\ldots+i_N\leqslant u+\kappa N-1}}\hspace{-5mm}\mathbb{P}(X_1=i_1)\mathbb{P}(X_2=i_2)\cdots\mathbb{P}(X_N=i_N)\,
\varphi\left(u+\kappa N-\sum_{j=1}^Ni_j\right).
\end{align}
Substituting $u=0$ into the recursive formula \eqref{eq:rec}, we notice that in order to find $\varphi(\ka N)$ we need to know the values of $\varphi(0),\, \varphi(1),\, \ldots,\, \varphi(\kappa N-1)$. Moreover, if we know the values of $\varphi(0),\, \varphi(1),\, \ldots,\, \varphi(\kappa N-1)$, the same recurrence \eqref{eq:rec} allows us to calculate $\f(u)$ for any $u\geqslant\ka N$ by substituting $u=0,\, 1, \ldots$ there. Thus, as mentioned in the introduction, we only need a way to calculate these initial values.

We now define a series of notations. Recalling that we aim to know the distribution of $\sup_{n\geq1}\sum_{i=1}^{n}\left(X_i-\ka\right)$, we define $N$ random variables:
\begin{align}\label{rvs_M}
\mathcal{M}_1:=\sup_{n\geqslant 1}\left(\sum_{i=1}^n(X_i-\ka)\right)^+,\,
\mathcal{M}_2:=\sup_{n\geqslant 2}\left(\sum_{i=2}^n(X_i-\ka)\right)^+,\, \ldots\,,\,
\mathcal{M}_N:=\sup_{n\geqslant N}\left(\sum_{i=N}^n(X_i-\ka)\right)^+\hspace{-0.2cm},
\end{align}
where $x^+=\max\{0,\,x\}$, $x\in\mathbb{R}$ is the positive part function. Obviously, same as $X_1,\,X_2,\,\ldots,\,X_N$, every random variable $\mathcal{M}_1,\,\mathcal{M}_2,\,\ldots,\,\mathcal{M}_N$ attains the values from the set \{0,\,1,\,\dots\}.

Let us denote the probability mass functions of $\mathcal{M}_j$, their generators $X_j$ and the sum $S_N=X_1+X_2+\ldots+X_N$:
\begin{equation}\label{0+}
m_i^{(j)}:=\mathbb{P}\left(\mathcal{M}_j=i\right),\ \  x_i^{(j)}:=\mathbb{P}\left(X_j=i\right),\ \ s_i^{(N)}:=\mathbb{P}\left(S_N=i\right),
\end{equation}
where $j\in\{1,\,2,\,\ldots,\,N\}$ and $i\in\mathbb{N}_0$. Let $F_{X_j}(i)$ be the distribution function of the random variable $X_j$, i. e.
$$
F_{X_j}(u)=\mathbb{P}(X_j\leqslant u)=\sum_{i=0}^{u}x^{(j)}_i,\,j\in\{1,\,2,\,\ldots,\,N\},\, u\in\mathbb{N}_0.
$$
Also, for the complex number $s\in\mathbb{C}$ let us denote the probability-generating function of some non-negative and integer-valued random variable $X$
$$
G_X(s)=\sum_{i=0}^{\infty}s^i\T(X=i)=\E s^X,\,|s|\leqslant1.
$$

The definition of the survival probability \eqref{ult_time_p} and the definition of random variable $\mathcal{M}_1$ imply
    \begin{equation}\label{eq:phi(n+1)}
        \varphi(u+1)=\mathbb{P}(\mathcal{M}_1\leqslant u)= \sum_{i=0}^{u}m^{(1)}_i\textit{ for all } u\in\mathbb{N}_0.
    \end{equation}
Thus, the survival probability calculation turns into the setup of the distribution function of $\mathcal{M}_1$. It is simple to explain the core idea of the paper, i. e. how the probabilities $m_i^{(1)}$, $i\in\mathbb{N}_0$ are attained. Let us refer to Feller's book \cite[Theorem on page 198]{Feller}. The referenced Theorem states that if $N=1$ in model \eqref{model}, i. e. the random walk $\{\sum_{i=1}^{n}X_i,\,n\in\mathbb{N}\}$ is generated by independent and identically distributed random variables, which are the copies of $X$, then $(\mathcal{M}_1+X-\kappa)^+\dist \mathcal{M}_1$. For arbitrary number of periods $N\in\mathbb{N}$ the mentioned distribution property is as follows: 
\begin{align}\label{dist_M}
(\mathcal{M}_1+X_N-\ka)^+\dist\mathcal{M}_N \text{ and } (\mathcal{M}_j+X_{j-1}-\ka)^+\dist \mathcal{M}_{j-1},
\text{ for all }j=2,\,3,\,\ldots,\,N,
\end{align}
see Lemma \ref{lem:dist_NxN} in Section \ref{sec:lemas} below. Metaphorically, the distributions' equalities in \eqref{dist_M} mean that the random variables $\mathcal{M}_1,\,\mathcal{M}_2,\,\ldots,\,\mathcal{M}_N$ ''can see each other'', and, more precisely, based on the equalities in \eqref{dist_M}, we can set up a system of corresponding equalities of probability-generating functions
\begin{align}\label{syst}
\begin{cases}
&\E s^{(\mathcal{M}_1+X_N-\ka)^+}=\E s^{\mathcal{M}_N}\\
&\E s^{(\mathcal{M}_2+X_1-\ka)^+}=\E s^{\mathcal{M}_1}\\
&\vdots\\
&\E s^{(\mathcal{M}_N+X_{N-1}-\ka)^+}=\E s^{\mathcal{M}_{N-1}}
\end{cases}.
\end{align}
The system \eqref{syst} contains the desired information on $m^{(1)}_i,\,i\in\mathbb{N}_0$. 

In general, the random variables $\mathcal{M}_1,\,\mathcal{M}_2,\,\ldots,\,\mathcal{M}_N$
can be extended, i. e. $\lim_{u\to\infty}\T(\mathcal{M}_j=u)>0$, $j=1,\,2,\,\ldots,\,N$. However, $\lim_{u\to\infty}\T(\mathcal{M}_j< u)=1$ for all $j=1,\,2,\,\ldots,\,N$ if $\E S_N<\ka N$, see Lemma \ref{lem:lim} in Section \ref{sec:lemas} below.

The condition $\E S_N<\ka N$ is called {\it the net profit condition} and it is crucially important for the survival probability $\vphi(u)$. Intuitively, an insurer has no chance to survive in long-term activity, if threatening amounts of claims on average are greater or equal to the collected premiums. This can also be well illustrated by the expected value of $W(n)$ in \eqref{model}. For instance,
$$
\E W(nN)=u+n(-\E S_N+\ka N)<0
$$
if $\E S_N>\ka N$ and $n$ is sufficiently large. Consequently, the negative value of $W(n)$ is unavoidable. See Theorem \ref{thm:2} in Section \ref{main} below for precise statements on the survival probability $\vphi(u),\,u\in\mathbb{N}_0$ when the net profit condition is violated, i. e. $\E S_N\geqslant\ka N$.

\section{Main results}\label{main}

In this section, based on the previously introduced notations and explanation that our goal is to know the probability mass function of $\mathcal{M}_1$, we formulate two main theorems for ultimate time survival probability calculation under the net profit condition. Theorem \ref{thm:1}, implied by system \eqref{syst}, provides the relations between $m^{(1)}_i,\,m^{(2)}_i,\,\ldots,\,m^{(N)}_i$ and $x^{(1)}_i,\,x^{(2)}_i,\,\ldots,\,x^{(N)}_i$ for all $i\in\mathbb{N}_0$ and lays down the foundation for ultimate time survival probability $\vphi(u+1)=\sum_{i=0}^{u}m^{(1)}_i,\,u\in\mathbb{N}_0$ calculation.

\begin{thm}\label{thm:1}
Suppose that the N-seasonal discrete-time risk model \eqref{model} is generated by random variables $X_1,\, X_2,\, \ldots,\, X_N$ and the net profit condition $\E S_N<\ka N$ is satisfied. Then the following statements are correct:
\begin{enumerate}
\item
The probability mass functions of random variables $\mathcal{M}_1,\,\mathcal{M}_2,\, \ldots,\, \mathcal{M}_N$ and $X_1,\,X_2,\, \ldots,\, X_N$ are related by
\begin{align}\label{eq:thm_eq_mean}
\nonumber
&\hspace{-1cm}\sum_{i=0}^{\kappa-1}m_i^{(1)}\sum_{j=0}^{\kappa-1-i}x^{(N)}_j(\kappa-i-j)+\sum_{i=0}^{\kappa-1}m_i^{(2)}\sum_{j=0}^{\kappa-1-i}x^{(1)}_j(\kappa-i-j)
+\sum_{i=0}^{\kappa-1}m_i^{(3)}\sum_{j=0}^{\kappa-1-i}x^{(2)}_j(\kappa-i-j)+\ldots\\
&\hspace{-1cm}+\sum_{i=0}^{\kappa-1}m_i^{(N)}\sum_{j=0}^{\kappa-1-i}x^{(N-1)}_j(\kappa-i-j)
=\kappa N -\mathbb{E}S_N.
\end{align}
\item
If $0<|s|\leqslant1$, then
\begin{align}\label{eq:thm_zero_s}
\nonumber
&\hspace{-1cm}\sum_{i=0}^{\kappa-1}m_i^{(1)}\sum_{j=0}^{\kappa-1-i}x^{(N)}_j(s^\kappa-s^{i+j})
+\frac{G_{X_N}(s)}{s^\kappa}\sum_{i=0}^{\kappa-1}m_i^{(2)}\sum_{j=0}^{\kappa-1-i}x^{(1)}_j(s^\kappa-s^{i+j})\\ \nonumber
&\hspace{-1cm}+\frac{G_{X_N+X_1}(s)}{s^{2\kappa}}
\sum_{i=0}^{\kappa-1}m_i^{(3)}\sum_{j=0}^{\kappa-1-i}x^{(2)}_j(s^\kappa-s^{i+j})+\ldots\\
&\hspace{-1cm}+\frac{G_{X_N+X_1+\ldots+X_{N-2}}(s)}{s^{\kappa(N-1)}}
\sum_{i=0}^{\kappa-1}m_i^{(N)}\sum_{j=0}^{\kappa-1-i}x^{(N-1)}_j(s^\kappa-s^{i+j})
=s^\ka G_{M_N}(s)\left(1-\frac{G_{S_N}(s)}{s^{\ka N}}\right)
\end{align}
and, if $\alpha\neq1$, $0<|\alpha|\leqslant1$, is a root of $G_{S_N}(s)=s^{\kappa N}$, then
\begin{align}\label{eq:thm_zero}
\nonumber
&\hspace{-1cm}\sum_{i=0}^{\kappa-1}m_i^{(1)}\sum_{j=i}^{\kappa-1}\al^j F_{X_N}(j-i)
+\frac{G_{X_N}(\alpha)}{\alpha^\kappa}\sum_{i=0}^{\kappa-1}m_i^{(2)}\sum_{j=i}^{\kappa-1}\al^j F_{X_1}(j-i)\\ \nonumber
&\hspace{-1cm}+\frac{G_{X_N+X_1}(\alpha)}{\alpha^{2\kappa}}
\sum_{i=0}^{\kappa-1}m_i^{(3)}\sum_{j=i}^{\kappa-1}\al^j F_{X_2}(j-i)+\ldots\\
&\hspace{-1cm}+\frac{G_{X_N+X_1+\ldots+X_{N-2}}(\alpha)}{\alpha^{\kappa(N-1)}}
\sum_{i=0}^{\kappa-1}m_i^{(N)}\sum_{j=i}^{\kappa-1}\al^j F_{X_{N-1}}(j-i)
=0.
\end{align}

\item
If $\alpha\neq1$, $0<|\alpha|\leqslant1$, is a root of $G_{S_N}(s)=s^{\kappa N}$ of multiplicity $r$, $r=2,\,3,\,...,\,\kappa N-1$, then
\begin{align}\label{eq:thm_derivative}
\nonumber
&\hspace{-1cm}\sum_{i=0}^{\kappa-1}m_i^{(1)}\sum_{j=i}^{\kappa-1}\frac{d^n}{ds^n}\left(s^j\right)\Big{|}_{s=\alpha}F_{X_N}(j-i)+
\sum_{i=0}^{\kappa-1}m_i^{(2)}\sum_{j=i}^{\kappa-1}\frac{d^n}{ds^n}\left(G_{X_N}(s)s^{j-\kappa}\right)\Big{|}_{s=\alpha} F_{X_1}(j-i)\\ \nonumber
&\hspace{-1cm}+
\sum_{i=0}^{\kappa-1}m_i^{(3)}\sum_{j=i}^{\kappa-1}\frac{d^n}{ds^n}\left(G_{X_N+X_1}(s)s^{j-2\ka}\right)\Big{|}_{s=\alpha} F_{X_2}(j-i)+\ldots\\
&\hspace{-1cm}+\sum_{i=0}^{\kappa-1}m_i^{(N)}\sum_{j=i}^{\kappa-1}\frac{d^n}{ds^n}\left(G_{X_N+X_1+\ldots+X_{N-2}}(s)s^{j-\kappa(N-1)}\right)\Big{|}_{s=\alpha} F_{X_{N-1}}(j-i)
=0,
\end{align}
for all $n\in\{1,\,2,\,\ldots,\,r-1\}$, where $\frac{d^n}{dt^n}(\ldots)|_{t=\alpha}$ denotes the $n'th$ derivative at $s=\alpha$.
\item
For $n=\kappa,\,\kappa+1,\,\ldots\,$, the probability mass functions of random variables $\mathcal{M}_1,\,\mathcal{M}_2,\, \ldots,\, \mathcal{M}_N$ and $X_1,\,X_2,\, \ldots,\, X_N$ satisfy the following system of equations
\begin{align}\label{eq:thm_system}
\begin{cases}
m^{(1)}_nx^{(N)}_0&=m^{(N)}_{n-\kappa}-\sum_{i=0}^{n-1}m^{(1)}_ix^{(N)}_{n-i}-\sum_{i=0}^{\kappa-1}m_i^{(1)}\sum_{j=0}^{\kappa-1-i}x^{(N)}_j\mathbbm{1}_{\{n=\kappa\}}\\
m^{(2)}_nx^{(1)}_0&=m^{(1)}_{n-\kappa}-\sum_{i=0}^{n-1}m^{(2)}_ix^{(1)}_{n-i}-\sum_{i=0}^{\kappa-1}m_i^{(2)}\sum_{j=0}^{\kappa-1-i}x^{(1)}_j\mathbbm{1}_{\{n=\kappa\}}\\
m^{(3)}_nx^{(2)}_0&=m^{(2)}_{n-\kappa}-\sum_{i=0}^{n-1}m^{(3)}_ix^{(2)}_{n-i}-\sum_{i=0}^{\kappa-1}m_i^{(3)}\sum_{j=0}^{\kappa-1-i}x^{(2)}_j\mathbbm{1}_{\{n=\kappa\}}\\
&\,\,\vdots\\
m^{(N)}_nx^{(N-1)}_0&=m^{(N-1)}_{n-\kappa}-\sum_{i=0}^{n-1}m^{(N)}_ix^{(N-1)}_{n-i}-\sum_{i=0}^{\kappa-1}m_i^{(N)}\sum_{j=0}^{\kappa-1-i}x^{(N-1)}_j\mathbbm{1}_{\{n=\kappa\}}\\
\end{cases}.
\end{align}
\end{enumerate}
\end{thm}

Let us comment on how Theorem \ref{thm:1} is used to obtain the distribution of $\mathcal{M}_1$, i. e. $m^{(1)}_i$, $i\in\mathbb{N}_0$. First, we note that the equation $s^{\ka N}=G_{S_N}(s)$ has exactly $\ka N-1$ roots in $|s|\leqslant1,\,s\neq1$ counted with their multiplicities, see Lemma \ref{lem:roots} in Section \ref{sec:lemas} below. We denote these roots by $\al_1,\,\al_2,\,\ldots,\,\al_{\ka N-1}$. We then create the system of linear equations \eqref{main_syst} replicating the equation \eqref{eq:thm_zero} $\ka N-1$ times over the roots $\al_1,\,\al_2,\,\ldots,\,\al_{\ka N-1}$ and include \eqref{eq:thm_eq_mean} as the last equation. To illustrate that explicitly, we define the matrices $M_1,\,M_2,\,\dots,\,M_N$ and $G_2,\,G_3,\,\dots,\,G_N$:

\begin{align*}
\hspace{-1cm}
M_1:=
\begin{pmatrix}
\sum_{j=0}^{\kappa-1}\al_1^jF_{X_N}(j)&\sum_{j=1}^{\kappa-1}\al_1^jF_{X_N}(j-1)&\ldots&\sum_{j=\ka-2}^{\kappa-1}\al_1^jF_{X_N}(j-\ka-2)&x_0^{(N)}\al_1^{\kappa-1}\\
\vdots&\vdots&\ddots&\vdots&\vdots\\
\sum_{j=0}^{\kappa-1}\al_{\ka N-1}^jF_{X_N}(j)&\sum_{j=1}^{\kappa-1}\al_{\ka N-1}^jF_{X_N}(j-1)&\ldots&\sum_{j=\ka-2}^{\kappa-1}\al_{\ka N-1}^jF_{X_N}(j-\ka-2)&x_0^{(N)}\al_{\ka N-1}^{\kappa-1}\\
\sum_{j=0}^{\kappa-1}x^{(N)}_j(\kappa-j)&\sum_{j=0}^{\kappa-2}x^{(N)}_j(\kappa-j-1)&\ldots&\sum_{j=0}^{1}x^{(N)}_j(2-j)&x_0^{(N)}
\end{pmatrix}_{\kappa N\times \kappa}\hspace{-0.5cm},
\end{align*}
\begin{align*}
\hspace{-1cm}
M_2:=
\begin{pmatrix}
\sum_{j=0}^{\kappa-1}\al_1^jF_{X_1}(j)&\sum_{j=1}^{\kappa-1}\al_1^jF_{X_1}(j-1)&\ldots&\sum_{j=\ka-2}^{\kappa-1}\al_1^jF_{X_1}(j-\ka-2)&x_0^{(1)}\al_1^{\kappa-1}\\
\vdots&\vdots&\ddots&\vdots&\vdots\\
\sum_{j=0}^{\kappa-1}\al_{\ka N-1}^jF_{X_1}(j)&\sum_{j=1}^{\kappa-1}\al_{\ka N-1}^jF_{X_1}(j-1)&\ldots&\sum_{j=\ka-2}^{\kappa-1}\al_{\ka N-1}^jF_{X_1}(j-\ka-2)&x_0^{(1)}\al_{\ka N-1}^{\kappa-1}\\
\sum_{j=0}^{\kappa-1}x^{(1)}_j(\kappa-j)&\sum_{j=0}^{\kappa-2}x^{(1)}_j(\kappa-j-1)&\ldots&\sum_{j=0}^{1}x^{(1)}_j(2-j)&x_0^{(1)}
\end{pmatrix}_{\kappa N\times \kappa}\hspace{-0.5cm},
\end{align*}
$$
\vdots
$$
\begin{align*}
&\hspace{-1cm}
M_N:=\\
&\hspace{-1cm}
\begin{pmatrix}
\sum_{j=0}^{\kappa-1}\al_1^jF_{X_{N-1}}(j)&\sum_{j=1}^{\kappa-1}\al_1^jF_{X_{N-1}}(j-1)&\ldots&\sum_{j=\ka-2}^{\kappa-1}\al_1^jF_{X_{N-1}}(j-\ka-2)&x_0^{(N-1)}\al_1^{\kappa-1}\\
\vdots&\vdots&\ddots&\vdots&\vdots\\
\sum_{j=0}^{\kappa-1}\al_{\ka N-1}^jF_{X_{N-1}}(j)&\sum_{j=1}^{\kappa-1}\al_{\ka N-1}^jF_{X_{N-1}}(j-1)&\ldots&\sum_{j=\ka-2}^{\kappa-1}\al_{\ka N-1}^jF_{X_{N-1}}(j-\ka-2)&x_0^{(N-1)}\al_{\ka N-1}^{\kappa-1}\\
\sum_{j=0}^{\kappa-1}x^{(N-1)}_j(\kappa-j)&\sum_{j=0}^{\kappa-2}x^{(N-1)}_j(\kappa-j-1)&\ldots&\sum_{j=0}^{1}x^{(N-1)}_j(2-j)&x_0^{(N-1)}
\end{pmatrix}_{\kappa N\times \kappa}\hspace{-0.5cm},
\end{align*}
\begin{align*}
G_2:=
\begin{pmatrix}
\frac{G_{X_N}(\al_1)}{\al_1^\ka}&\ldots&\frac{G_{X_N}(\al_1)}{\al_{1}^\ka}\\
\vdots&\ddots&\vdots\\
\frac{G_{X_N}(\al_{\kappa N-1})}{\al_{\kappa N-1}^\ka}&\ldots&\frac{G_{X_N}(\al_{\kappa N-1})}{\al_{\kappa N-1}^\ka}\\
1&\ldots&1
\end{pmatrix}_{\kappa N\times \kappa}\hspace{-0.5cm},  \quad 
G_3:=
\begin{pmatrix}
\frac{G_{X_N+X_1}(\al_1)}{\al_1^{2\ka}}&\ldots&\frac{G_{X_N+X_1}(\al_1)}{\al_1^{2\ka}}\\
\vdots&\vdots&\vdots\\
\frac{G_{X_N+X_1}(\al_{\kappa N-1})}{\al_{\kappa N-1}^{2\ka}}&\ldots&\frac{G_{X_N+X_1}(\al_{\kappa N-1})}{\al_{\kappa N-1}^{2\ka}}\\
1&\ldots&1
\end{pmatrix}_{\kappa N\times \kappa}\hspace{-0.5cm},\,\ldots, 
\end{align*}

\begin{align*}
G_N:=
\begin{pmatrix}
\frac{G_{X_N+X_1+\ldots+X_{N-2}}(\al_1)}{\al_1^{\ka(N-1)}}&\ldots&\frac{G_{X_N+X_1+\ldots+X_{N-2}}(\al_1)}{\al_1^{\ka(N-1)}}\\
\vdots&\ddots&\vdots\\
\frac{G_{X_N+X_1+\ldots+X_{N-2}}(\al_{\kappa N-1})}{\al_{\kappa N-1}^{\ka(N-1)}}&\ldots&\frac{G_{X_N+X_1+\ldots+X_{N-2}}(\al_{\kappa N-1})}{\al_{\kappa N-1}^{\ka(N-1)}}\\
1&\ldots&1
\end{pmatrix}_{\kappa N\times \kappa}\hspace{-0.5cm},
\end{align*}
and set up the system
\begin{align}\label{main_syst}
\begin{pmatrix}
M_1&M_2\circ G_2&\ldots&M_N \circ G_N
\end{pmatrix}_{\kappa N\times \kappa N}
\times
\begin{pmatrix}
m^{(1)}_0\\
m^{(1)}_1\\
\vdots\\
m^{(1)}_{\kappa-1}\\
m^{(2)}_0\\
m^{(2)}_1\\
\vdots\\
m^{(2)}_{\kappa-1}\\
\vdots\\
m^{(N)}_0\\
m^{(N)}_1\\
\vdots\\
m^{(N)}_{\kappa-1}
\end{pmatrix}_{\kappa N\times1}
=
\begin{pmatrix}
0\\
\vdots\\
0\\
\kappa N-\mathbb{E}S_N
\end{pmatrix}_{\kappa N\times1}\hspace{-0.5cm},
\end{align}
where $\circ$ denotes the Hadamard matrix product, also known as the element-wise product, entry-wise product or Schur product, i. e. two elements in corresponding positions in two matrices are multiplied, see for example \cite{horn_johnson}.

Clearly, the solution of \eqref{main_syst} (see Section \ref{notes} below on solvability and modifications of \eqref{main_syst}) gives $m^{(1)}_0,\,m^{(1)}_1,\,\ldots,\,m^{(1)}_{\ka-1}$, while using the system \eqref{eq:thm_system} (note that we can have $m^{(j)}_0,\,m^{(j)}_1,\,\ldots,\,m^{(j)}_{\ka-1}$, $j\in\{2,\,3,\,\ldots,\,N\}$ from \eqref{main_syst} too) we can calculate $m^{(1)}_{\ka}$, $m^{(1)}_{\ka+1}$, $\ldots$, $m^{(1)}_{\ka N-1}$ and consequently $\vphi(1)$, $\vphi(2)$, $\dots$, $\vphi(\ka N)$. Having $\vphi(1),\,\vphi(2),\,\ldots,\,\vphi(\ka N)$ we can obtain $\vphi(0)$ by setting $u=0$ in recurrence \eqref{eq:rec} and use the same recurrence \eqref{eq:rec} to proceed calculating $\vphi(\ka N+1),\,\vphi(\ka N+2),\,\ldots$ Of course, the survival probabilities $\vphi(\ka N+1),\,\vphi(\ka N+2),\,\ldots$ can be calculated by system \eqref{eq:thm_system} too. Moreover, we can set up the ultimate time survival probability-generating function. Let
\begin{align}
\Xi(s):=\sum_{i=0}^{\infty}\vphi(i+1)s^i,\,|s|<1.
\end{align}
In view of \eqref{eq:phi(n+1)}, it is easy to observe that 
\begin{align}\label{eq:gen_f_relation}
\Xi(s)=\sum_{i=0}^{\infty}\vphi(i+1)s^i=\sum_{i=0}^{\infty}s^i\sum_{j=0}^{i}m^{(1)}_j
=\sum_{j=0}^{\infty}m^{(1)}_j\frac{s^j}{1-s}
=\frac{G_{\mathcal{M}_1}(s)}{1-s},\,|s|<1.
\end{align}
Then, the ultimate time survival probability-generating function $\Xi(s)$ admits the following expression.

\begin{thm}\label{thm:gen}
Let us assume that the probabilities $m^{(j)}_0,\,m^{(j)}_1,\,\ldots,\,m^{(j)}_{\ka-1}$, $j\in\{1,\,2,\,\ldots,\,N\}$ are known beforehand. Then, the survival probability-generating function $\Xi(s)$, for $|s|<1$ and $s^{\ka N}\neq G_{S_N}(s)$ admits the following representation
\begin{align*}
\hspace{-0.5cm}
\Xi(s)=\frac{\mathbbm{u}^T\mathbbm{v}}{G_{S_N}(s)-s^{\ka N}},\textit{ where }
\mathbbm{u}=\begin{pmatrix}
s^{\ka(N-1)}\\s^{\ka(N-2)}G_{X_1}(s)\\s^{\ka(N-3)}G_{X_1+X_2}(s)\\ \vdots\\s^\ka G_{X_1+X_2+\ldots+X_{N-2}}(s)\\G_{X_1+X_2+\ldots+X_{N-1}}(s)
\end{pmatrix},\,
\mathbbm{v}=
\begin{pmatrix}
&\sum_{i=0}^{\kappa-1}m_i^{(2)}\sum_{j=i}^{\kappa-1}s^j F_{X_1}(j-i)\\
&\sum_{i=0}^{\kappa-1}m_i^{(3)}\sum_{j=i}^{\kappa-1}s^j F_{X_2}(j-i)\\
&\sum_{i=0}^{\kappa-1}m_i^{(4)}\sum_{j=i}^{\kappa-1}s^j F_{X_3}(j-i)\\
&\vdots\\
&\sum_{i=0}^{\kappa-1}m_i^{(N)}\sum_{j=i}^{\kappa-1}s^j F_{X_{N-1}}(j-i)\\
&\sum_{i=0}^{\kappa-1}m_i^{(1)}\sum_{j=i}^{\kappa-1}s^j F_{X_N}(j-i)\\
\end{pmatrix}.
\end{align*}
\end{thm}
The next theorem states that if the net profit condition is unsatisfied, then ultimate time survival is impossible except in some cases when $S_N$ is degenerate.

\begin{thm}\label{thm:2}
    Suppose that N-seasonal discrete-time risk model \eqref{model} is generated by random variables $X_1, X_2, \ldots, X_N$ and the net profit condition is not satisfied. In this case:
    \begin{enumerate}
        \item
            If $\E S_N > \kappa N$, then $\varphi(u)=0$ for all $u\in\mathbb{N}_0$.
        \item
            If $\E S_N = \kappa N$ and $\mathbb{P}(S_N=\kappa N)<1$, then $\varphi(u)=0$ for all $u\in\mathbb{N}_0$.
        \item
            If $\mathbb{P}(S_N=\kappa N)=1$, then random variables $X_1, X_2, \ldots, X_N$ are degenerate and
            \begin{align*}
                &u+\kappa n^*-\sum_{k=1}^{n^*}X_k\leqslant0 \Rightarrow \f(u)=0,\\
                &u+\kappa n^*-\sum_{k=1}^{n^*}X_k>0 \Rightarrow \f(u)=1,
            \end{align*}
            here $n^*$ is equal to such $n\in\{1, 2, \ldots, N\}$ with which $\kappa n-\sum_{k=1}^{n}X_k$ attains its minimum.
   \end{enumerate}
\end{thm}

The last theorem provides an algorithm for the calculation of finite time survival probability $\vphi(u,\,T)$. Let us define 
\begin{align}\label{Jon_def}
\vphi^{(j)}(u,\,T)=\mathbb{P}\left(\sup_{1\leqslant n \leqslant T }\sum_{i=1}^{n}\left(X_i^{(j)}-\ka\right)<u\right),
\end{align}
where $j\in\{1,\,2,\,\ldots,\,N\}$ and $X_i^{(j)}=X_{i+j-1}$. It is easy to observe that $\vphi^{(N+j)}(u,\,T)=\vphi^{(j)}(u,\,T)$.

\begin{thm}\label{thm:fin_t}
For the finite time survival probability \eqref{fin_time_p} of the N-seasonal discrete time risk model defined in \eqref{model}, holds:
\begin{align}\label{eq:fin_t_no_rec}
&\varphi(u,1)=\sum_{i\leqslant u+\ka-1}x^{(1)}_{i},\,
\varphi(u,2)=\hspace{-2mm}\sum_{\substack{i_1\leqslant u+\ka-1\\i_1+i_2\leqslant u+2\ka-1}}\hspace{-2mm}x^{(1)}_{i_1}x^{(2)}_{i_2},\,
\ldots,\,\varphi(u,\, N)=\hspace{-6mm}\sum_{\substack{i_1\leqslant u+\ka-1\\i_1+i_2\leqslant u+2\ka-1\\ \vdots\\i_1+i_2+\ldots+i_N\leqslant u+\ka N-1}} \hspace{-8mm}x^{(1)}_{i_1}x^{(2)}_{i_2}\cdots x^{(N)}_{i_N},
\end{align}
and
\begin{align}\label{eq:fin_t_rec}
\varphi(u,T)=\hspace{-6mm}\sum_
{\substack{i_1\leqslant u+\ka-1\\i_1+i_2\leqslant u+2\ka-1\\ \vdots\\i_1+i_2+\ldots+i_N\leqslant u+\ka N-1}}
\hspace{-8mm}x^{(1)}_{i_1}x^{(2)}_{i_2}\cdots x^{(N)}_{i_N}\varphi\left(u+\ka N-i_1-\ldots-i_N,\,T- N\right),
\end{align}
if $T\in\{ N+1,\, N+2,\,\ldots\}$. 

Moreover, for the defined probabilities \eqref{Jon_def}, it holds that   
\begin{align}
&\vphi^{(j)}(u,\,1)=F_{X_j}(u+\ka-1), \label{Jon}\\
&\vphi^{(j)}(u,\,T)=\sum_{i=0}^{u+\ka-1}\vphi^{(j+1)}(u+\ka-i,\,T-1)x^{(j)}_i,\,T=2,\,3,\,\ldots\label{Jon1}
\end{align}
\end{thm}

The formulated Theorems \ref{thm:1}, \ref{thm:gen}, \ref{thm:2} and \ref{thm:fin_t} are proved in Section \ref{sec:proofs} below.

\section{Notes on the solution of linear system involving probabilities of $\mathcal{M}_1$, $\mathcal{M}_2$, $\ldots$, $\mathcal{M}_N$}\label{notes}

In general, it is not easy to give an explicit solution of system \eqref{main_syst} or even to prove that system's determinant $\ka N \times \ka N$ never equals zero. For instance, if $N=1$ and $\ka\in\mathbb{N}$, then the system \eqref{main_syst} is

\begin{align}\label{eq:system_N_1}
\begin{pmatrix}
\sum_{j=0}^{\kappa-1}\al_1^jF_X(j)&\sum_{j=0}^{\kappa-2}\al_1^{j+1}F_X(j)&\ldots&\al_1^{\kappa-1}x_0\\
\vdots&\vdots&\ddots&\vdots\\
\sum_{j=0}^{\kappa-1}\al_{\ka-1}^jF_X(j)&\sum_{j=0}^{\kappa-2}\al_{\ka-1}^{j+1}F_X(j)&\ldots&\al_{\ka-1}^{\kappa-1}x_0\\
\sum_{j=0}^{\kappa-1}x_j(\kappa-j)&\sum_{j=0}^{\kappa-2}x_j(\kappa-j-1)&\ldots&x_0
\end{pmatrix}
\begin{pmatrix}
&m^{(1)}_0\\
&\vdots\\
&m^{(1)}_{\kappa-2}\\
&m^{(1)}_{\kappa-1}
\end{pmatrix}
=
\begin{pmatrix}
&0\\
&\vdots\\
&0\\
&\kappa-\mathbb{E}X
\end{pmatrix},
\end{align}
where $X\dist X_1$, $x_i=x_i^{(1)}$ and $\al_1,\,\al_2,\,\ldots,\,\al_{\ka-1}\neq1$ are the simple roots of $s^\ka=G_X(s)$ in $|s|\leqslant1$. Then, if $x_0>0$, determinant of the system's matrix in \eqref{eq:system_N_1} is Vandermonde-like
$$
\frac{x_0^{\kappa}}{(-1)^{\kappa+1}}\prod_{j=1}^{\kappa-1}(\al_j-1)\prod_{1\leqslant i<j\leqslant \ka-1}(\al_j-\al_i)\neq0
$$
and the probabilities $m^{(1)}_0,\,m^{(1)}_1,\,\ldots,\,m^{(1)}_{\kappa-1}$ together with the survival probabilities $\vphi(0),\,\vphi(1),\,\ldots,\,\vphi(\ka)$ admit neat closed-form expressions, see \cite[Thm. 4]{Grig}. On the other hand, if $\ka=1$ and $N\in\mathbb{N}$, then the system \eqref{main_syst} is

\begin{align}\label{eq:system_k_1}
\begin{pmatrix}
x_0^{(N)}&\frac{x_0^{(1)}G_{X_N}(\al_1)}{\al_1}&\frac{x_0^{(2)}G_{X_N+X_1}(\al_1)}{\al^2_1}&\ldots&
\frac{x_0^{(N-1)}G_{X_N+X_1+\ldots+X_{N-2}}(\al_1)}{\al^{N-1}_1}\\
x_0^{(N)}&\frac{x_0^{(1)}G_{X_N}(\al_2)}{\al_2}&\frac{x_0^{(2)}G_{X_N+X_1}(\al_2)}{\al^2_2}&\ldots&
\frac{x_0^{(N-1)}G_{X_N+X_1+\ldots+X_{N-2}}(\al_2)}{\al^{N-1}_2}\\
\vdots&\vdots &\vdots & \ddots&\vdots\\
x_0^{(N)}&\frac{x_0^{(1)}G_{X_N}(\al_{N-1})}{\al_{N-1}}&\frac{x_0^{(2)}G_{X_N+X_1}(\al_{N-1})}{\al^2_{N-1}}&\ldots&
\frac{x_0^{(N-1)}G_{X_N+X_1+\ldots+X_{N-2}}(\al_{N-1})}{\al^{N-1}_{N-1}}\vspace{2mm}\\
x_0^{(N)}&x_0^{(1)}&x_0^{(2)}&\ldots&x_0^{(N-1)}
\end{pmatrix}\begin{pmatrix}
m_0^{(1)}\\
m_0^{(2)}\\
\vdots\\
m_0^{(N-1)}\\
m_0^{(N)}
\end{pmatrix}
=
\begin{pmatrix}
0\\
0\\
\vdots\\
0\\
N-\mathbb{E}S_N
\end{pmatrix},
\end{align}
where $\al_1,\,\al_2,\,\ldots,\,\al_{N-1}\neq1$ are the simple roots of $s^N=G_{S_N}(s)$ in $|s|\leqslant1$, see \cite{GJS}. If $N=3$, the main matrix in \eqref{eq:system_k_1} is 

\begin{align}
A:=\begin{pmatrix}
x_0^{(3)}&\frac{x_0^{(1)}G_{X_3}(\al_1)}{\al_1}&\frac{x_0^{(2)}G_{X_3+X_1}(\al_1)}{\al^2_1}\\
x_0^{(3)}&\frac{x_0^{(1)}G_{X_3}(\al_2)}{\al_2}&\frac{x_0^{(2)}G_{X_3+X_1}(\al_2)}{\al^2_2}\\
x_0^{(3)}&x_0^{(1)}&x_0^{(2)}
\end{pmatrix}
\end{align}
and one may check that for $s^{(3)}_0=\T(X_1+X_2+X_3=0)>0$ the matrix $A$ is non-singular iff 
$$
\left(\frac{G_{X_3}(\al_1)}{\al_1}-1\right)
\left(\frac{G_{X_3+X_1}(\al_2)}{\al_2^2}-1\right)
\neq 
\left(\frac{G_{X_3}(\al_2)}{\al_2}-1 \right)
\left(\frac{G_{X_3+X_1}(\al_1)}{\al_1^2}-1\right)
$$ 
where $\al_1,\,\al_2\neq1$ are the simple roots of $s^3=G_{X_1+X_2+X_3}(s)$ in $0<|s|\leqslant 1$. Computer calculations with some chosen random variables $X_1,\,X_2,\,\ldots,\,X_N$, $N\geqslant 3$  
do not reveal any examples that the system's matrix in \eqref{main_syst} is singular. The listed thoughts raise the following conjecture.
\begin{con}
Assume that $s^{(N)}_0=\T(X_1+X_2+\ldots+X_N=0)>0$ and $\al_1,\,\al_2,\,\ldots,\,\al_{\ka N-1}\neq1$ are the simple roots of $s^{\ka N}=G_{S_N}(s)$. Then, the system's matrix in \eqref{main_syst} is non-singular for all $\ka,\,N\in\mathbb{N}$. In particular, if $N=3$ and $\ka=1$, then
$$
\frac{G_{S_3}(\al_1)}{\al_1^3}=\frac{G_{S_3}(\al_2)}{\al_2^3}=1
$$
implies
$$
\left(\frac{G_{X_3}(\al_1)}{\al_1}-1\right)
\left(\frac{G_{X_3+X_1}(\al_2)}{\al_2^2}-1\right)
\neq 
\left(\frac{G_{X_3}(\al_2)}{\al_2}-1 \right)
\left(\frac{G_{X_3+X_1}(\al_1)}{\al_1^2}-1\right)
$$ 
and consequently $\det A\neq0$, where $\al_1,\,\al_2\neq1$ locate in $|s|\leqslant1$.
\end{con}

    Let us comment on how the system \eqref{main_syst} gets modified if there are non-simple roots among $\al_1,\,\al_2,\,\ldots,\,\al_{\ka N-1}$ and/or the random variable $S_N$ does not attain its ''small'' values. Clearly, $\T \left(S_N\geqslant j\right)=1$ for some $j\in\{1,\,2,\,\ldots,\,\ka N-1\}$ implies at least one column of zeros in the main matrix of \eqref{main_syst}. Note that $\T \left(S_N\geqslant \ka N\right)=1$ violates the net profit condition $\E S_N<\ka N$. Also, $\T \left(S_N\geqslant j\right)=1$ for some $j\in\{1,\,2,\,\ldots,\,\ka N-1\}$ always implies fewer terms in the right-hand-side of recurrence \eqref{eq:rec} because some probability mass functions equal zero then. For example, if $\T(X_N\geqslant \ka N-1)=1$, then $\vphi(0)=s^{(N)}_{\ka N-1}\vphi(1)$ and there is only $\vphi(0)$ we must know to use the 
    recurrence \eqref{eq:rec} in order to find $\vphi(1),\,\vphi(2),\,\ldots$ Thus, when $\T \left(S_N\geqslant j\right)=1$ for some $j\in\{1,\,2,\,\ldots,\,\ka N-1\}$, we have to adjust the main matrix in \eqref{main_syst} according to the equalities \eqref{eq:thm_zero} and \eqref{eq:thm_eq_mean} not including any columns of zeros.

    Also, when some roots of $s^{\ka N}=G_{S_N}(s)$ are of multiplicity $r\in\{2,\,3,\,\ldots,\,\ka N-1\}$, then, to avoid identical lines in \eqref{main_syst}, we must replace the corresponding lines by derivatives as provided in equality \eqref{eq:thm_derivative}. 
    
    Once again, computational examples with some chosen random variables $X_1,\,X_2,\,\ldots,\,X_N,$ $\,N\geqslant 3$ and $\kappa\geqslant1$ do not reveal any examples that such a modified (due to multiple roots and/or $S_N$ not attaining ''small'' values) system's matrix in \eqref{main_syst} is singular.

\section{Lemmas}\label{sec:lemas}

In this section, we formulate and prove several auxiliary lemmas that are later used to prove theorems formulated in Section \ref{main}. Some of the present lemmas are direct generalizations of statements from \cite[Sec. 5]{GJS}, where they are proved for $X_j-1$, $j\in\{1,\,2,\,\ldots,\, N\}$, while here we need them for $X_j-\ka$, $j\in\{1,\,2,\,\ldots,\, N\}$, $\ka\in\mathbb{N}$.

\begin{lem}\label{lem:lim}
    If the net profit condition is satisfied, i. e. $\E S_N<\ka N$, then
	$$
	    \lim_{u\to\infty}\mathbb{P}(\mathcal{M}_j<u)=1,
	$$
for all $j\in\{1,\,2,\,\ldots,\,N\}$.
\end{lem}
\begin{proof}
    We prove the case $j=1$ only and note that the other cases can be proven similarly.
	
	According to the law of large numbers, almost surely
	\begin{align*}
	    \frac{1}{n}\sum_{i=1}^{n}\left(X_{i}-\ka\right)&=\frac{1}{N}\left(\frac{N}{n}\sum_{\substack{{i=1}\\ i\equiv 1\,{\rm mod}\,N}}^{n}\left(X_{i}-\ka\right) +\ldots+\frac{N}{n}\sum_{\substack{{i=1}\\ i\equiv N\,{\rm mod}\,N}}^{n}\left(X_{i}-\ka\right)\right)\\
	    &\overset{n \to \infty}{\longrightarrow} \frac{1}{N}\left((\mathbb{E}X_1-\ka)+\ldots+(\mathbb{E}X_N-\ka)\right)
	    =\frac{\mathbb{E}S_N-\ka N}{N}=:-\mu<0.
	\end{align*}
	Therefore, 	\begin{align*}
	\lim_{n\to\infty}\mathbb{P}\left(\sup_{j\geq n}\left|\frac{1}{j}\sum_{i=1}^{j}\left(X_i-\ka\right)+\mu\right| < \frac{\mu}{2}\right) = 1.
	\end{align*}
	Consequently, for any arbitrarily small $\varepsilon>0$, there exists such number $N_\varepsilon\in\mathbb{N}$ that
	\begin{align*}
	    &\mathbb{P}\left(\bigcap_{j=n}^{\infty}\left\{\sum_{i=1}^{j}\left(X_i-\ka\right)\leqslant
	    0\right\}\right)	\geqslant
	    \mathbb{P}\left(\bigcap_{j=n}^{\infty}\left\{\sum_{i=1}^{j}\left(X_i-\ka\right)<
	    -\frac{\mu}{2}\right\}\right)\\
	    &\geqslant\mathbb{P}\left(\bigcap_{j= n}^{\infty}\left\{\left|\frac{1}{j}\sum_{i=1}^{j}\left(X_i-\ka\right)+\mu\right| < \frac{\mu}{2}\right\}\right)
	    = \mathbb{P}\left(\sup_{j\geqslant n}\left|\frac{1}{j}\sum_{i=1}^{j}\left(X_i-\ka\right)+\mu\right| < \frac{\mu}{2}\right) \geqslant 1-\varepsilon
	\end{align*}
 for all $n \geqslant N_\varepsilon$.
 
	It follows that for any such $\varepsilon$ and any $u\in\mathbb{N}$ we have
	\begin{align*}
	    \mathbb{P}(\mathcal{M}_1< u)&=\mathbb{P}\left(\bigcap_{j=1}^{\infty}\left\{\sum_{i=1}^{j}\left(X_i-\ka\right)<u\right\}\right)\\
	    &\geqslant\mathbb{P}\left(\left\{\bigcap_{j=1}^{N_\varepsilon-1}\left\{\sum_{i=1}^{j}\left(X_i-\ka\right)<u\right\}\right\}\cap\left\{\bigcap_{j=N_\varepsilon}^{\infty}\left\{\sum_{i=1}^{j}\left(X_i-\ka\right)\leqslant0\right\}\right\}\right)\\
	    &\geqslant\mathbb{P}\left(\bigcap_{j=1}^{N_\varepsilon-1}\left\{\sum_{i=1}^{j}\left(X_i-\ka\right)<u\right\}\right)
     +\mathbb{P}\left(\bigcap_{j=N_\varepsilon}^{\infty}\left\{\sum_{i=1}^{j}\left(X_i-\ka\right)\leqslant0\right\}\right)-1\\
	    &\geqslant\mathbb{P}\left(\bigcap_{j=1}^{N_\varepsilon-1}\left\{\sum_{i=1}^{j}\left(X_i-\ka\right)<u\right\}\right)-\varepsilon.
	\end{align*}
	The last inequality implies
	\begin{equation*}
	    \lim_{u \to \infty}\mathbb{P}(\mathcal{M}_1<u)\geqslant 1-\varepsilon,
	\end{equation*}
	where $\varepsilon>0$ is as small as we want and the assertion of Lemma follows.
\end{proof}

\begin{lem}\label{lem:dist_NxN}
    If the net profit condition is satisfied, i. e. $\E S_N<\ka N$, it holds that $(\mathcal{M}_j+X_{j-1}-\ka)^+\dist \mathcal{M}_{j-1}$, for all $j=2,\,3,\,\ldots,\,N$, and $(\mathcal{M}_1+X_N-\ka)^+\dist \mathcal{M}_N$.
\end{lem}

\begin{proof}
We prove the equality $(\mathcal{M}_1+X_N-\ka)^+\dist \mathcal{M}_N$ only and note that the other ones can be proved by the same arguments. According to Lemma \ref{lem:lim}, the random variable $\mathcal{M}_1$ is not extended, i. e. it does not attain infinity. Let us denote $\tilde{X}_j=X_j-\ka$ for all $j\in\{1,\,2,\,\ldots,\,N\}$. Then

\begin{align*}
&(\mathcal{M}_1+\tilde{X}_N)^+\\
&\dist(\max\{0,\,\max\{\tilde{X}_{1},\,\tilde{X}_{1}+\tilde{X}_{2},\,\tilde{X}_{1}+\tilde{X}_{2}+\tilde{X}_{3},\,\ldots\}\}+\tilde{X}_N)^+\\
&\dist(\max\{0,\,\tilde{X}_{1},\,\tilde{X}_{1}+\tilde{X}_{2},\,\tilde{X}_{1}+\tilde{X}_{2}+\tilde{X}_{3},\,\ldots\}+\tilde{X}_N)^+\\
&\dist(\max\{\tilde{X}_{N},\,\tilde{X}_{N}+\tilde{X}_{1},\,\tilde{X}_{N}+\tilde{X}_{1}+\tilde{X}_{2},\,\tilde{X}_{N}+\tilde{X}_{1}+\tilde{X}_{2}+\tilde{X}_3,\,\ldots\})^+\\
&\dist(\max\{\tilde{X}_{N},\,\tilde{X}_{N}+\tilde{X}_{N+1},\,\tilde{X}_{N}+\tilde{X}_{N+1}+\tilde{X}_{N+2},\,\tilde{X}_{N}+\tilde{X}_{N+1}+\tilde{X}_{N+2}+\tilde{X}_{N+3},\,\ldots\})^+\\
&\dist\max\{0,\max\{\tilde{X}_{N},\,\tilde{X}_{N}+\tilde{X}_{N+1},\,\tilde{X}_{N}+\tilde{X}_{N+1}+\tilde{X}_{N+2},\,\ldots\}\}\\
&\dist \mathcal{M}_N^+.
\end{align*}
\end{proof}

\begin{lem}\label{lem:generating_f_NxN}
Let $s\in\mathbb{C}$ be the complex number and $0<|s|\leqslant1$. Then the probability-generating functions of $X_1,\,X_2,\,\ldots,\,X_N$ and $\mathcal{M}_1,\,\mathcal{M}_2\ldots,\,\mathcal{M}_N$ are related the following way
\begin{align}\label{syst:NxN}
    \begin{cases}
    \sum_{i=0}^{\kappa-1}m_i^{(1)}\sum_{j=0}^{\kappa-1-i}x^{(N)}_j(s^\kappa-s^{i+j})&=s^\kappa G_{\mathcal{M}_N}(s)-G_{X_N}(s)G_{\mathcal{M}_1}(s)\\
    \sum_{i=0}^{\kappa-1}m_i^{(2)}\sum_{j=0}^{\kappa-1-i}x^{(1)}_j(s^\kappa-s^{i+j})&=s^\kappa G_{\mathcal{M}_1}(s)-G_{X_1}(s)G_{\mathcal{M}_2}(s)\\
    \sum_{i=0}^{\kappa-1}m_i^{(3)}\sum_{j=0}^{\kappa-1-i}x^{(2)}_j(s^\kappa-s^{i+j})&=s^\kappa G_{\mathcal{M}_2}(s)-G_{X_2}(s)G_{\mathcal{M}_3}(s)\\
    &\,\,\vdots\\
    \sum_{i=0}^{\kappa-1}m_i^{(N)}\sum_{j=0}^{\kappa-1-i}x^{(N-1)}_j(s^\kappa-s^{i+j})&=s^\kappa G_{\mathcal{M}_{N-1}}(s)-G_{X_{N-1}}(s)G_{\mathcal{M}_N}(s)
    \end{cases}.
\end{align}
\end{lem}

\begin{proof}
Let us demonstrate how the first equality in \eqref{syst:NxN} is derived and note that the remaining ones follow the same logic.

By Lemma \ref{lem:dist_NxN}, the distributions' equality $(\mathcal{M}_1+X_{N}-\ka)^+\dist \mathcal{M}_{N}$ implies the equality of probability-generating functions $G_{\mathcal{M}_N}(s)=G_{(\mathcal{M}_1+X_{N}-\ka)^+}(s)$. Then, applying the law of total expectation for the last equality, we obtain
\begin{align*}
G_{\mathcal{M}_N}(s)
&=\mathbb{E} s^{(\mathcal{M}_1+X_{N}-\ka)^+}=\mathbb{E} \left( \mathbb{E} \left( s^{(\mathcal{M}_1+X_{N}-\ka)^+}|\mathcal{M}_1\right) \right)
=\sum_{i=0}^{\infty}m^{(1)}_i\E s^{(X_N-\kappa+i)^+}\\
&=\sum_{i=0}^{\kappa-1}m^{(1)}_i\mathbb{E}s^{(X_N-\kappa+i)^+}
+G_{X_N}(s)s^{-\kappa}\sum_{i=\kappa}^{\infty}m^{(1)}_is^i\\
&=\sum_{i=0}^{\kappa-1}m^{(1)}_i\left(\mathbb{E}s^{\left(X_N-\kappa+i\right)^+}
-s^{i-\kappa}G_{X_N}(s)\right)+s^{-\kappa}G_{X_N}(s)G_{\mathcal{M}_1}(s).
\end{align*}
Multiplying both sides of the last equality by $s^\kappa$ when $s\neq0$ and observing that
\begin{align*}
\sum_{i=0}^{\kappa-1}m^{(1)}_i\left(s^\ka\E s^{(X_N-\kappa+i)^+}-s^iG_{X_N}(s)\right)=\sum_{i=0}^{\kappa-1}m^{(1)}_i
\sum_{j=0}^{\ka-1-i}x_j^{(N)}(s^\ka-s^{i+j})
\end{align*}
we get the desired result.
\end{proof}

The next lemma provides the quantity and location of the roots of $s^{\ka N}=G_{S_N}(s)$.

\begin{lem}\label{lem:roots}
Assume that the net profit condition $G'_{S_N}(1)=\E S_N<\ka N$ is valid. Then there are exactly $\ka N-1$ roots, counted with their multiplicities, of $s^{\ka N}=G_{S_N}(s)$ in $|s|\leqslant1$, $s\neq1$.
\end{lem}

\begin{proof}
We follow the proof of \cite[Lemma 9]{GJ}. Due to the estimate 
$$
|G_{S_N}(s)|\leqslant1<\lambda |s|^{\ka N}
$$
on $|s|=1$ when $\lambda>1$, Rouché's theorem implies that both functions $G_{S_N}(s)-\lambda s^{\ka N}$ and $\la s^{\ka N}$ have the same number of roots in $|s|<1$ and this number is $\ka N$ due to the fundamental theorem of algebra.  When $\la\to1^+$ some roots of $G_{S_N}(s)-\lambda s^{\ka N}$ remain in $|s|<1$ and some migrate to the boundary points $|s|=1$. Obviously, $s=1$ is always the root of $s^{\ka N}=G_{S_N}(s)$ and it is the simple root because the net profit condition holds, i. e. 

$$
\left(G_{S_N}(s)-s^{\ka N}\right)'\Big|_{s=1}=\E S_N-\ka N<0.
$$ 

Thus, there remain $\ka N-1$ roots of $s^{\ka N}=G_{S_N}(s)$ in $|s|\leqslant1$, $s\neq1$ and additionally one can say that $|s|=1$ , $s \neq 1$ is the root of $s^{\ka N}=G_{S_N}(s)$ if the greatest common divisor of $\ka N$ and the powers of $s$ in $G_{S_N}(s)$ is greater than one.

\end{proof}

\section{Proofs}\label{sec:proofs}
In this section, we provide the proofs for four theorems formulated in Section \ref{main}. 
\begin{proof}[Proof of Theorem \ref{thm:1}]
    We first prove equality \eqref{eq:thm_zero_s}. To derive \eqref{eq:thm_zero_s}, we use the system of equations \eqref{syst:NxN} from Lemma \ref{lem:generating_f_NxN}. According to the conditions of Lemma \ref{lem:generating_f_NxN}, $s\ne0$ and we rearrange the system \eqref{syst:NxN} by multiplying its first equality by $1$, the second one by $G_{X_N}(s)/s^\kappa$, the third one by 
    $G_{X_N+X_1}(s)/s^{2\kappa}$ and so on till the last equality which we multiply by $G_{X_N+X_1+\ldots+X_{N-2}}(s)/s^{\kappa(N-1)}$. We then add up all these equations and obtain
\begin{align}\label{eq:main_eq}
\nonumber
&\hspace{-1cm}\sum_{i=0}^{\kappa-1}m_i^{(1)}\sum_{j=0}^{\kappa-1-i}x^{(N)}_j(s^\kappa-s^{i+j})+
\frac{G_{X_N}(s)}{s^\kappa}\sum_{i=0}^{\kappa-1}m_i^{(2)}\sum_{j=0}^{\kappa-1-i}x^{(1)}_j(s^\kappa-s^{i+j})\\
\nonumber
&\hspace{-1cm}+\frac{G_{X_N+X_1}(s)}{s^{2\kappa}}
\sum_{i=0}^{\kappa-1}m_i^{(3)}\sum_{j=0}^{\kappa-1-i}x^{(2)}_j(s^\kappa-s^{i+j})+\ldots+\frac{G_{X_N+X_1+\ldots+X_{N-2}}(s)}{s^{\kappa(N-1)}}
\sum_{i=0}^{\kappa-1}m_i^{(N)}\sum_{j=0}^{\kappa-1-i}m^{(N-1)}_j(s^\kappa-s^{i+j})
\\
\nonumber
&\hspace{-1cm}=s^\kappa G_{\mathcal{M}_N}(s)-G_{X_N}(s)G_{\mathcal{M}_1}(s)+
\frac{G_{X_N}(s)}{s^\kappa}\left(s^\kappa G_{\mathcal{M}_1}(s)-G_{X_1}(s)G_{\mathcal{M}_2}(s)\right)\\ \nonumber
&\hspace{-1cm}+\frac{G_{X_N+X_1}(s)}{s^{2\kappa}}\left(s^\kappa G_{\mathcal{M}_2}(s)-G_{X_2}(s)G_{\mathcal{M}_3}(s)\right)
+\ldots+\frac{G_{X_N+X_1+\ldots+X_{N-2}}(s)}{s^{\kappa(N-1)}}\left(s^\kappa G_{\mathcal{M}_{N-1}}(s)
- G_{X_{N-1}}(s)G_{\mathcal{M}_N}(s)\right)\\ 
&\hspace{-1cm}=s^\kappa G_{\mathcal{M}_N}(s)-\frac{G_{S_N}(s)}{s^{\kappa(N-1)}}G_{\mathcal{M}_N}(s) =s^{\kappa}G_{\mathcal{M}_N}(s)\left(1-\frac{G_{S_N}(s)}{s^{\kappa N}}\right).
\end{align}
Here we have used the fact that if any random variables $X$ are $Y$ independent, then
\begin{equation*}
G_{X+Y}(s) = G_{X}(s) G_{Y}(s).
\end{equation*}
Thus, the equality \eqref{eq:thm_zero_s} is proved. We now derive \eqref{eq:thm_zero}.

It is obvious that the right-hand-side of \eqref{eq:main_eq} equals zero if we set $s=\alpha$, where $\alpha\neq1$ is the root of $G_{S_N}(s)=s^{\kappa N},\,|s|\leqslant1$. We then divide both sides of \eqref{eq:main_eq} by $\al-1$, i. e. 
$$
\frac{\al^\ka-\al^{i+j}}{\al-1}=\al^{j+i}+\al^{j+i+1}+\ldots+\al^{\ka-1},
\,\al\neq1,
$$	
and get
\begin{align*}
&\sum_{i=0}^{\kappa-1}m_i^{(1)}\sum_{j=0}^{\kappa-1-i}x^{(N)}_j\sum_{l=j+i}^{\ka-1}\al^l+
\frac{G_{X_N}(s)}{s^\kappa}\sum_{i=0}^{\kappa-1}m_i^{(2)}\sum_{j=0}^{\kappa-1-i}x^{(1)}_j\sum_{l=j+i}^{\ka-1}\al^l\\
\nonumber
&+\frac{G_{X_N+X_1}(s)}{s^{2\kappa}}
\sum_{i=0}^{\kappa-1}m_i^{(3)}\sum_{j=0}^{\kappa-1-i}x^{(2)}_j\sum_{l=j+i}^{\ka-1}\al^l+\ldots\\
&+\frac{G_{X_N+X_1+\ldots+X_{N-2}}(s)}{s^{\kappa(N-1)}}
\sum_{i=0}^{\kappa-1}m_i^{(N)}\sum_{j=0}^{\kappa-1-i}x^{(N-1)}_j\sum_{l=j+i}^{\ka-1}\al^l\\
&=\sum_{i=0}^{\kappa-1}m_i^{(1)}\sum_{j=i}^{\kappa-1}\al^j F_{X_N}(j-i)
+\frac{G_{X_N}(\alpha)}{\alpha^\kappa}\sum_{i=0}^{\kappa-1}m_i^{(2)}\sum_{j=i}^{\kappa-1}\al^j F_{X_1}(j-i)\\
&+\frac{G_{X_N+X_1}(\alpha)}{\alpha^{2\kappa}}
\sum_{i=0}^{\kappa-1}m_i^{(3)}\sum_{j=i}^{\kappa-1}\al^j F_{X_2}(j-i)+\ldots\\
&+\frac{G_{X_N+X_1+\ldots+X_{N-2}}(\alpha)}{\alpha^{\kappa(N-1)}}
\sum_{i=0}^{\kappa-1}m_i^{(N)}\sum_{j=i}^{\kappa-1}\al^j F_{X_{N-1}}(j-i)
=0,
\end{align*}
which is the claimed equality \eqref{eq:thm_zero}.

We now consider the equation \eqref{eq:thm_derivative}. Since $s\neq1$, we divide the both sides of \eqref{eq:main_eq} by $s-1$ and rewrite  the right-hand-side of it as follows
$$
s^{\kappa(1-N)}G_{\mathcal{M}_N}(s)\left(s^{\kappa N}-G_{S_N}(s)\right)/(s-1).
$$
Clearly, derivatives
$$
\frac{d^n}{ds^n}\left(s^{\kappa(1-N)}G_{\mathcal{M}_N}(s)\left(s^{\kappa N}-G_{S_N}(s)\right)/(s-1)\right)\Big|_{s=\al}=0
$$
for all $n\in\{1,\,2,\,\ldots,\,r-1\}$ if $\al\neq1$ is the root of $s^{\ka N}=G_{S_N}(s),\,|s|\leqslant1$ and root's multiplicity is $r\in\{2,\,3,\,\ldots,\,\ka N-1\}$.
Thus, the equality \eqref{eq:thm_derivative} is nothing but the $n$'th derivative of \eqref{eq:thm_zero_s} at $s=\al$.
    
    We now prove equality \eqref{eq:thm_eq_mean} in Theorem \ref{thm:1}.
    The $s$ derivative of both sides of equation \eqref{eq:main_eq} is
   
\begin{align}\label{eq:pirmu_isvestiniu_suma}\nonumber
&\hspace{-0.75cm}\sum_{i=0}^{\kappa-1}m_i^{(1)}\sum_{j=0}^{\kappa-1-i}x^{(N)}_j(\kappa s^{\kappa-1}-(i+j)s^{i+j-1})
+\sum_{i=0}^{\kappa-1}m_i^{(2)}\sum_{j=0}^{\kappa-1-i}x^{(1)}_j(\kappa s^{\kappa-1}-(i+j)s^{i+j-1})\\ \nonumber
&\hspace{-0.75cm}+\sum_{i=0}^{\kappa-1}m_i^{(3)}\sum_{j=0}^{\kappa-1-i}x^{(2)}_j(\kappa s^{\kappa-1}-(i+j)s^{i+j-1})+\ldots\\
&\hspace{-0.75cm}+\sum_{i=0}^{\kappa-1}m_i^{(N)}\sum_{j=0}^{\kappa-1-i}x^{(N-1)}_j(\kappa s^{\kappa-1}-(i+j)s^{i+j-1})
\nonumber\\
&\hspace{-0.75cm}=G'_{\mathcal{M}_N}(s)\left(s^\ka-G_{S_N}(s)s^{\ka(1-N)}\right)
+G_{\mathcal{M}_N}(s)\left(\ka s^{\ka-1}-G_{S_N}(s)\ka(1-N)s^{\ka(1-N)-1}-G'_{S_N}(s)s^{\ka(1-N)}\right).
\end{align}
	We continue the proof by letting $s\to1^-$ in \eqref{eq:pirmu_isvestiniu_suma}. Because the net profit condition $\E S_N<\ka N$ holds and the random variable $\mathcal{M}_N\dist(\mathcal{M}_1+X_N-\ka)^+$ attains finite values only, we obtain
\begin{align}\label{riba1}\nonumber
&\sum_{i=0}^{\kappa-1}m_i^{(1)}\sum_{j=0}^{\kappa-1-i}x^{(N)}_j(\kappa -(i+j))
+\sum_{i=0}^{\kappa-1}m_i^{(2)}\sum_{j=0}^{\kappa-1-i}x^{(1)}_j(\kappa -(i+j))\\ \nonumber
&+\sum_{i=0}^{\kappa-1}m_i^{(3)}\sum_{j=0}^{\kappa-1-i}x^{(2)}_j(\kappa -(i+j))+\ldots
+\sum_{i=0}^{\kappa-1}m_i^{(N)}\sum_{j=0}^{\kappa-1-i}x^{(N-1)}_j(\kappa-(i+j))
\nonumber\\
&=\lim_{s\to1^-}G'_{\mathcal{M}_N}(s)\left(s^\ka-G_{S_N}(s)s^{\ka(1-N)}\right)
+\ka N-\E S_N.
\end{align}
	
Calculating the limit in \eqref{riba1} there are two separate cases to examine: $\E \mathcal{M}_N<\infty$ and $\E \mathcal{M}_N=\infty$. If $\E \mathcal{M}_N<\infty$, then
the limit in \eqref{riba1} is zero. However, this limit is zero even if  $\E \mathcal{M}_N=\infty$. Indeed, if $\E \mathcal{M}_N=\infty$, then by using L'Hospital's rule we get
\begin{align*}
&\lim_{s\to1^-}G'_{\mathcal{M}_N}(s)\left(s^\ka-G_{S_N}(s)s^{\ka(1-N)}\right)
=\lim_{s\to1^-}\frac{s^\ka-G_{S_N}(s)s^{\ka(1-N)}}{1/G'_{\mathcal{M}_N}(s)}\nonumber
=\lim_{s\to1^-}\frac{\left(s^\ka-G_{S_N}(s)s^{\ka(1-N)}\right)'}{\left(1/G'_{\mathcal{M}_N}(s)\right)'}\nonumber\\
&=\lim_{s\to1^-}\left(\ka s^{\ka-1}-G'_{S_N}(s)s^{\ka(1-N)}-G_{S_N}(s)\ka(1-N)s^{\ka(1-N)-1}\right)\frac{(G'_{\mathcal{M}_N}(s))^2} {-G''_{\mathcal{M}_N}(s)}=(\ka N-\E S_N)\cdot0=0,
\end{align*}
because 
\begin{align}\label{Fedor}
\lim_{s\to1^-}\frac{(G'_{\mathcal{M}_N}(s))^2}{G''_{\mathcal{M}_N}(s)}=0,
\end{align}
see \cite[Lem. 5.5]{GJS}\footnote{Proof's direction of \eqref{Fedor} was originally provided by Fedor Petrov.}. Thus, the limit in \eqref{riba1} is zero and the equality \eqref{eq:thm_eq_mean} in Theorem \ref{thm:1} follows.

It remains to prove the equalities in system \eqref{eq:thm_system}. In short, every equality in system \eqref{eq:thm_system} is the corresponding equality from \eqref{syst:NxN} expanded at $s=0$. Let us demonstrate the derivation of the first equality in \eqref{eq:thm_system} in detail and note that the remaining ones are derived analogously. We need to show that the first equality in \eqref{syst:NxN} 
\begin{align}\label{eq:gen_implicant}
\sum_{i=0}^{\kappa-1}m_i^{(1)}\sum_{j=0}^{\kappa-1-i}x^{(N)}_j(s^\kappa-s^{i+j})=s^\kappa G_{\mathcal{M}_N}(s)-G_{X_N}(s)G_{\mathcal{M}_1}(s)
\end{align}
implies (the first one in \eqref{eq:thm_system})
$$
m^{(1)}_nx^{(N)}_0=m^{(N)}_{n-\kappa}-\sum_{i=0}^{n-1}m^{(1)}_ix^{(N)}_{n-i}-\sum_{i=0}^{\kappa-1}m_i^{(1)}\sum_{j=0}^{\kappa-1-i}x^1_j\mathbbm{1}_{\{n=\kappa\}},
\,n=\ka,\,\ka+1,\,\ldots
$$
or, equivalently, 
\begin{align}\label{eq:4}
\sum_{i=0}^{\kappa-1}m_i^{(1)}\sum_{j=0}^{\kappa-1-i}x^{(N)}_j\mathbbm{1}_{\{n=\kappa\}} = m_{n-\kappa}^{(N)}-\sum_{i=0}^n m_i^{(1)}x_{n-i}^{(N)},\,n=\ka,\,\ka+1,\,\ldots
\end{align}
Equality \eqref{eq:4} is implied by \eqref{eq:gen_implicant} because of the following equalities: 
\begin{align*}
&\frac{1}{n!}\frac{d^n}{ds^n}\left(\sum_{i=0}^{\ka-1}m_i^{(1)}\sum_{j=0}^{\ka-1-i}x_j^{(N)}\left(s^\ka-s^{i+j}\right)\right)\Bigg|_{s=0}=\sum_{i=0}^{\kappa-1}m_i^{(1)}\sum_{j=0}^{\kappa-1-i}x^{(N)}_j\mathbbm{1}_{\{n=\kappa\}},\\
&\frac{1}{n!}\frac{d^n}{ds^n}\left(s^\ka G_{\mathcal{M}_N}(s)\right)\big|_{s=0}=m^{(N)}_{n-\ka},\\
&\frac{1}{n!}\frac{d^n}{ds^n}\left(G_{\mathcal{M}_1}(s)G_{X_N}(s)\right)\big|_{s=0}=\sum_{i=0}^n m_i^{(1)}x_{n-i}^{(N)},
\end{align*}
when $n=\ka,\,\ka+1,\,\ldots$ The proof of Theorem \ref{thm:1} is finished.
\end{proof}

\begin{proof}[Proof of Theorem \ref{thm:gen}]
Let us rewrite the system \eqref{syst:NxN} the following way

\begin{align}\label{syst:NxN_m}\nonumber
&
\begin{pmatrix}
s^\ka&-G_{X_1}(s)&0&\ldots&0&0\\
0&s^\ka&-G_{X_2}(s)&\ldots&0&0\\
\vdots&\vdots&\vdots&\ddots&\vdots&\vdots\\
0&0&0&\ldots&s^\ka&-G_{X_{N-1}}(s)\\
-G_{X_N}(s)&0&0&\ldots&0&s^\ka
\end{pmatrix}
\begin{pmatrix}
G_{\mathcal{M}_1}(s)\\
G_{\mathcal{M}_2}(s)\\
\vdots\\
G_{\mathcal{M}_{N-1}}(s)\\
G_{\mathcal{M}_N}(s)
\end{pmatrix}\\
&\hspace{8cm}=
\begin{pmatrix}
\sum_{i=0}^{\kappa-1}m_i^{(2)}\sum_{j=0}^{\kappa-1-i}x^{(1)}_j(s^\kappa-s^{i+j})\\
\sum_{i=0}^{\kappa-1}m_i^{(3)}\sum_{j=0}^{\kappa-1-i}x^{(2)}_j(s^\kappa-s^{i+j})\\
\vdots\\
\sum_{i=0}^{\kappa-1}m_i^{(N)}\sum_{j=0}^{\kappa-1-i}x^{(N-1)}_j(s^\kappa-s^{i+j})\\
\sum_{i=0}^{\kappa-1}m_i^{(1)}\sum_{j=0}^{\kappa-1-i}x^{(N)}_j(s^\kappa-s^{i+j})
\end{pmatrix}
\end{align}
and denote this system by $AB=C$. Determinant of the main matrix in \eqref{syst:NxN_m} is
$$
\det A=s^{\ka N}-G_{S_N}(s).
$$
Thus, the main matrix in \eqref{syst:NxN_m} is invertible for all such $s$ that $s^{\ka N}\neq G_{S_N}(s)$ and $B=A^{-1}C$. Therefore, the previous thoughts and equality \eqref{eq:gen_f_relation} imply

$$
\Xi(s)=\frac{G_{\mathcal{M}_1}(s)}{1-s}=\frac{1}{s^{\ka N}-G_{S_N}(s)}
\begin{pmatrix}
M_{11}&M_{21}&\ldots&M_{N1}
\end{pmatrix}
\frac{C}{1-s},
$$
where $M_{11},\,M_{21},\,\ldots,\,M_{N1}$ are the minors of $A$ and $C$ is the column vector of the right-hand-side of \eqref{syst:NxN_m}.

\end{proof}

\begin{proof}[Proof of Theorem \ref{thm:2}]
We first show that $\E S_N > \kappa N$ implies $\varphi(u)=0$ for all $u\in\mathbb{N}_0$. The recurrence \eqref{eq:rec} yields
\begin{align}\label{sum_supl} \nonumber
\f(u)&=\sum_{\substack{i_1\leqslant u+\kappa-1\\i_1+i_2\leqslant u+2\kappa-1\\
\vdots \vspace{1mm} \\
i_1+i_2+\ldots +i_{N-1}\leqslant u+\ka(N-1)-1\\
i_1+i_2+\ldots+i_N\leqslant u+\kappa N-1}}\hspace{-10mm}\mathbb{P}(X_1=i_1)\mathbb{P}(X_2=i_2)\cdots\mathbb{P}(X_N=i_N)\,
\varphi\Bigg(u+\kappa N-\sum_{j=1}^Ni_j\Bigg)\\ \nonumber
&=\sum_{i=1}^{u+\ka N}s^{(N)}_{u+\ka N -i}\f(i)\\
&\quad- \sum_{\substack{i_1\leqslant u+\kappa-1\\i_1+i_2\leqslant u+2\kappa-1\\
\vdots \vspace{1mm} \\
u+\ka (N-1)\leqslant i_1+i_2+\ldots +i_{N-1}\leqslant u+\ka N-1\\ \nonumber
i_1+i_2+\ldots+i_N\leqslant u+\kappa N-1}}\hspace{-10mm}x^{(1)}_{i_1}x^{(2)}_{i_2}\cdots x^{(N)}_{i_N}\,
\varphi\Bigg(u+\kappa N-\sum_{j=1}^Ni_j\Bigg)\\ \nonumber
&\quad\quad\vdots\\ \nonumber
&\quad- \sum_{\substack{u+\ka\leqslant i_1\leqslant u+\kappa N-1\\i_1+i_2\leqslant u+\kappa N-1\\
\vdots \vspace{1mm} \\
i_1+i_2+\ldots +i_{N-1}\leqslant u+\kappa N-1\\
i_1+i_2+\ldots+i_N\leqslant u+\kappa N-1}}\hspace{-0mm}x^{(1)}_{i_1}x^{(2)}_{i_2}\cdots x^{(N)}_{i_N}\,
\varphi\Bigg(u+\kappa N-\sum_{j=1}^Ni_j\Bigg)\\ 
&=\sum_{i=1}^{u+\ka N}s^{(N)}_{u+\ka N -i}\f(i) - \sum_{i=1}^{\ka(N-1)} \mu_i(u)\f(i),
\end{align}
where $\mu_i(u)$ for each $i\in\{1,\,2,\, \ldots,\, \ka (N-1)\}$ are coefficients consisting of the products of probability mass functions of random variables $X_1, X_2, \ldots, X_N$. For instance, if $N=2$ and $\ka=1$, then
$$
\vphi(u)=\sum_{\substack{i_1\leqslant u\\i_1+i_2\leqslant u+1}}x^{(1)}_{i_1}x^{(2)}_{i_2}\vphi(u+2-i_1-i_2)=\sum_{i=1}^{u+2}s^{(2)}_{u+2-i}\vphi(i)-x^{(1)}_{u+1}x^{(2)}_0\vphi(1).
$$
If $\mu_0(u):=s^{(N)}_{u+\ka N}$ and $\mu_j(u):=0$ when $j>\ka (N-1)$, then the equality in \eqref{sum_supl} is
    \begin{equation*}
        \f(u)=\sum_{i=0}^{u+\ka N}s^{(N)}_{u+\ka N -i}\f(i) - \sum_{i=0}^{\ka N-1} \mu_i(u)\f(i).
    \end{equation*}
    By summing up both sides of the last equality by $u$, which varies from 0 to some natural and sufficiently large $v$, we obtain
    \begin{equation}\label{eq:sum}
        \sum_{u=0}^v\f(u)=\sum_{u=0}^v\sum_{i=0}^{u+\ka N}s^{(N)}_{u+\ka N -i}\f(i) - \sum_{u=0}^v\sum_{i=0}^{\ka N-1} \mu_i(u)\f(i).
    \end{equation}
    We now change the order of summation in \eqref{eq:sum}
    \begin{equation*}
        \sum_{u=0}^v\sum_{i=0}^{u+\ka N}(\cdot)=\sum_{i=0}^{\ka N-1}\sum_{u=0}^v(\cdot)+\sum_{i=\ka N}^{v+\ka N}\sum_{u=i-\ka N}^v(\cdot)
    \end{equation*}
    and obtain
    \begin{align*}
        \sum_{u=0}^{v+\ka N}\f(u)-\sum_{u=v+1}^{v+\ka N}\f(u)
        =\sum_{i=0}^{\ka N-1}\f(i)\sum_{u=0}^v s^{(N)}_{u+\ka N -i}+\sum_{i=\ka N}^{v+\ka N}\f(i)\sum_{u=i-\ka N}^v s^{(N)}_{u+\ka N -i} - \sum_{i=0}^{\ka N-1}\f(i)\sum_{u=0}^v \mu_i(u).
    \end{align*}
    Subtracting $\sum_{i=0}^{v+\ka N}\f(i) \sum_{u=i-\ka N}^v s_{u+\ka N-i}^{(N)}$ from both sides of the last equation and rearranging, we get
    \begin{align*}
        &\sum_{i=0}^{v+\ka N}\f(i)\left(1 -\sum_{u=i-\ka N}^v s_{u+\ka N-i}^{(N)}\right)-\sum_{i=v+1}^{v+\ka N}\f(i) \\
        &=\sum_{i=0}^{\ka N-1}\f(i)\left(\sum_{u=0}^v s^{(N)}_{u+\ka N -i}-\sum_{u=0}^v \mu_i(u)\right)
        -\sum_{i=0}^{\ka N-1}\f(i) \sum_{u=i-\ka N}^v s_{u+\ka N-i}^{(N)}
    \end{align*}
    or
    \begin{align*}
        \sum_{i=v+1}^{v+\ka N}\f(i)-\sum_{i=0}^{v+\ka N}\f(i)\left(1 -\sum_{u=0}^{v+\ka N-i} s_u^{(N)}\right)
        =\sum_{i=0}^{\ka N-1}\f(i)\left(\sum_{u=0}^{\ka N -i-1} s^{(N)}_{u}+\sum_{u=0}^{v} \mu_i(u)\right),
    \end{align*}
    which implies
    \begin{align}\label{eq:v_to_inf}\hspace{-0.5cm}
    \sum_{i=v+1}^{v+\ka N}\f(i)-\sum_{i=0}^{v+\ka N}\f(i)\mathbb{P}(S_N>v+\ka N - i)=\sum_{i=0}^{\ka N-1}\f(i)\left(\mathbb{P}(S_N\leqslant \ka N-i-1)+\sum_{u=0}^{v} \mu_i(u)\right).
    \end{align}
    Clearly, the definition of the survival probability \eqref{ult_time_p} implies that $\f(u)$ is a non-decreasing function, i. e. $\f(u)\leqslant\f(u+1)$ for all $u\in\mathbb{N}_0$. Thus, there exists a non-negative limit $\f(\infty):=\lim_{u\to\infty}\f(u)$ and $\f(\infty)=1$ if the net profit condition $\E S_N<\ka N$ holds, see Lemma \ref{lem:lim}. We now let $v\to\infty$ in both sides of \eqref{eq:v_to_inf}. For the first sum in \eqref{eq:v_to_inf} we obtain
    \begin{equation}\label{2lim1}
        \lim_{v\to\infty}\sum_{i=v+1}^{v+\ka N}\f(i)=\lim_{v\to\infty}\left(\f(v+1)+\ldots+\f(v+\ka N)\right)=\f(\infty) \ka N,
    \end{equation}
 and for the second
    \begin{align}\label{eq:the_second}
\lim_{v\to\infty}\sum_{i=0}^{v+\ka N}\f(i)\mathbb{P}(S_N>v+\ka N - i)=\f(\infty)\E S_N.
    \end{align}
   Indeed, let us recall that $\E X=\sum_{i=0}^\infty\mathbb{P}(X>i)$, when $X$ is some non-negative and integer-valued random variable. Then, the upper bound of \eqref{eq:the_second} is
    \begin{align*}
        &\lim_{v\to\infty}\sum_{i=0}^{v+\ka N}\f(i)\mathbb{P}(S_N>v+\ka N - i)\leqslant \lim_{v\to\infty}
        \f(v+\ka N)\sum_{i=0}^{v+\ka N}\mathbb{P}(S_N>v+\ka N-i)\\
        &=\lim_{v\to\infty}\f(v+\ka N)\sum_{i=0}^{v+\ka N}\mathbb{P}(S_N>i)=\f(\infty)\E S_N,
    \end{align*}
   while the lower bound is the same due to inequality
    \begin{align*}
        &\sum_{i=0}^{v+\ka N}\f(i)\mathbb{P}(S_N>v+\ka N - i)\\
        &=\sum_{j=0}^M\f(i)\mathbb{P}(S_N>v+\ka N - i)
        +\sum_{i=M+1}^{v+\ka N}\f(i)\mathbb{P}(S_N>v+\ka N - i)\\
        &\geqslant \inf_{i\geqslant M+1}\f(i)\sum_{i=0}^{v+\ka N-M-1}\mathbb{P}(S_N>i),
    \end{align*}
where $M$ is some fixed and sufficiently large natural number. Thus, when $v\to\infty$, the equality in \eqref{eq:v_to_inf} is

\begin{align}\label{eq:after_limit}
\f(\infty)(\ka N-\E S_N)=\sum_{i=0}^{\ka N-1}\f(i)\left(\mathbb{P}(S_N\leqslant \ka N-i-1)+\sum_{u=0}^{\infty} \mu_i(u)\right).
\end{align}
If $\E S_N>\ka N$, as under the case of consideration, the non-negative right-hand-side of \eqref{eq:after_limit} implies $\f(\infty)=0$ and consequently $\f(u)=0$ for all $u\in\mathbb{N}_0$. The first case that $\E S_N>\ka N$ makes survival impossible is proved.
        
	Let us now consider the case when $\E S_N = \kappa N$ and $\mathbb{P}(S_N=\kappa N)<1$. If $\E S_N = \kappa N$ and $\mathbb{P}(S_N=\kappa N)<1$, then at least one probability $s_0^{(N)},\, s_1^{(N)},\, \ldots,\, s_{\ka N-1}^{(N)}$ is larger than zero, because otherwise $\E S_N>\ka N$.
	Then, from \eqref{eq:after_limit},
    \begin{equation}\label{eq:3.2.2}
        \sum_{i=0}^{\ka N-1}\varphi(i)\left(\sum_{u=0}^{\ka N-i-1}s_u^{(N)}+\sum_{u=0}^{\infty} \mu_i(u)\right)=0.
    \end{equation}
    If $s_0^{(N)}>0$, then \eqref{eq:3.2.2} implies $\varphi(0)=\varphi(1)=\ldots=\varphi(\ka N-1)=0$. Using recurrence \eqref{eq:rec} we can show that $\varphi(u)=0$ for all $u \in \mathbb{N}_0$.
    If $s_0^{(N)}=0$ and $s_1^{(N)}>0$, then \eqref{eq:3.2.2} implies that $\varphi(0)=\varphi(1)=\ldots=\varphi(\ka N-2)=0$ and once again, using recurrence \eqref{eq:rec} we can show that $\varphi(u)=0$ for all $u \in \mathbb{N}_0$.
    Arguing the same we proceed up to $s_0^{(N)}=s_1^{(N)}=\ldots= s_{\ka N-2}^{(N)}=0,\, s_{\ka N-1}^{(N)}>0$ and observe that in all of these cases \eqref{eq:3.2.2} and recurrence \eqref{eq:rec} yields $\varphi(u)=0$ for all $u \in \mathbb{N}_0$.
    
	Finally, let us consider the case when $\mathbb{P}(S_N=\kappa N)=1$. If $S_N=\ka N$ with probability one, the random variables $X_1,\, X_2,\, \ldots,\, X_N$ are degenerate, meaning that $X_i \equiv c_i$ for all $i \in \{1,\, 2,\, \ldots,\, N\}$, where $c_i \in \{0,\, 1,\,\ldots,\, \ka N\}$ and $c_1+c_2+\ldots+c_N=\ka N$. Thus, the model \eqref{model} becomes fully deterministic. Moreover, $W(n)=W(n+N)$ for all $n \in \mathbb{N}_0$, $N\in\mathbb{N}$, and it is sufficient to check if the lowest of values among $W(1),\,\ldots,\,W(N)$ is larger than zero.
\end{proof}

\begin{proof}[Proof of Theorem \ref{thm:fin_t}]
The proof of equalities \eqref{eq:fin_t_no_rec} and \eqref{eq:fin_t_rec} is basically the same as deriving the recurrence \eqref{eq:rec}. Equalities \eqref{Jon} and \eqref{Jon1} are implied by \cite[Thm. 1]{BBS}.
\end{proof}

\section{Numerical examples}

In this section, we illustrate the applicability of theorems formulated in Section \ref{main}. All of the necessary calculations are performed using Wolfram Mathematica \cite{Mathematica}. Notice that some of the examples considered here are also considered in \cite[Sec. 4]{AG22}, where the ultimate time survival probability was obtained by calculating the limits of certain recurrent sequences. Therefore, in some examples here we check if the obtained values of $\vphi(u)$ match the previously known ones.
    
We say that a random variable $X$ is distributed according to the {\it displaced Poisson distribution} $\mathcal{P}(\lambda, \xi)$ with parameters $\lambda>0$ and $\xi\in\mathbb{N}_0$, if
	\begin{equation*}
	    \mathbb{P}(X=m)=e^{-\lambda}\frac{\lambda^{m-\xi}}{(m-\xi)!}, \quad m=\xi, \xi+1, \ldots
	\end{equation*}
	One can check that the following facts for the displaced Poisson distribution are true:
	\begin{enumerate}
	    \item $\E X = \lambda + \xi$; \eqnum\label{8.1}
	    \item 
	    If $X\sim\mathcal{P}(\lambda_1, \xi_1)$, $Y\sim\mathcal{P}(\lambda_2, \xi_2)$ and also $X$ and $Y$ are independent, then
	    \\$X+Y\sim\mathcal{P}(\lambda_1+\lambda_2, \xi_1+\xi_2)$; \eqnum\label{8.2} 
	\item $G_X(s)=s^\xi e^{\lambda(s-1)}$.
	\end{enumerate}

	\begin{ex}\label{ex:1}
Let $\ka=2$ and consider the bi-seasonal (N=2) discrete time risk model \eqref{model} where $X_1\sim\mathcal{P}(1, 0)$ and $X_2\sim\mathcal{P}(2, 0)$.
We set up the survival probability-generating function $\Xi(s)$ and calculate $\vphi(u)$, when $u=0,\,1,\,\ldots,\,15$. 	    
\end{ex}
	
Let us observe that in the considered example the net profit condition is satisfied $\E S_2=3<4$. Solving the equation $G_{S_2}(s)=e^{3(s-1)}=s^4$ for $|s|\leqslant1$, $s\neq1$, we get $\alpha_1:=-0.3605, \alpha_2:=-0.1294+0.4087i, \alpha_3:=-0.1294-0.4087i$.
    Since all of the solutions $\al_1,\,\al_2,\,\al_3$ are simple and none of them are equal to 0, following the description beneath Theorem \ref{thm:1} in Section \ref{main}, we set up matrices $M_1, M_2$ and $G_2$:
\begin{align*}\hspace{-1cm}
M_1=
\begin{pmatrix}
x^{(2)}_1\al_1+x^{(2)}_0(\al_1+1) & x^{(2)}_0\al_1 \\
x^{(2)}_1\al_2+x^{(2)}_0(\al_2+1) & x^{(2)}_0\al_2\\
x^{(2)}_1\al_3+x^{(2)}_0(\al_3+1) & x^{(2)}_0\al_3\\
x^{(2)}_1+2x^{(2)}_0 & x^{(2)}_0
\end{pmatrix},\,
M_2=
\begin{pmatrix}
x^{(1)}_1\al_1+x^{(1)}_0(\al_1+1) & x^{(1)}_0\al_1 \\
x^{(1)}_1\al_2+x^{(1)}_0(\al_2+1) & x^{(1)}_0\al_2\\
x^{(1)}_1\al_3+x^{(1)}_0(\al_3+1) & x^{(1)}_0\al_3\\
x^{(1)}_1+2x^{(1)}_0 & x^{(1)}_0
\end{pmatrix},\,
G_2=
\begin{pmatrix}
\frac{G_{X_2}(\al_1)}{\al^2_1} & \frac{G_{X_2}(\al_1)}{\al^2_1}\\
\frac{G_{X_2}(\al_2)}{\al^2_2} & \frac{G_{X_2}(\al_2)}{\al^2_2}\\			\frac{G_{X_2}(\al_3)}{\al^2_3} & \frac{G_{X_2}(\al_3)}{\al^2_3}\\       
1 & 1
\end{pmatrix}
\end{align*}
and the system
\begin{align*}
\begin{pmatrix}
M_1 & M_2\circ G_2
\end{pmatrix}_{4\times4}
\begin{pmatrix}
m^{(1)}_0\\
m^{(1)}_1\\
m^{(2)}_0\\
m^{(2)}_1
\end{pmatrix}
=
\begin{pmatrix}
0\\
0\\
0\\
1
\end{pmatrix},
\end{align*}
which implies $m^{(1)}_0=0.6501$, $m^{(1)}_1=0.1395$, $m^{(2)}_0=0.5083$, $m^{(2)}_1=0.1855$. Then, $\varphi(1)=m^{(1)}_0=0.6501$ and $\f(2)=m^{(1)}_0+m^{(1)}_1=0.7896$. We then can use the system  \eqref{eq:thm_system} to find $m^{(1)}_2, m^{(1)}_3, \ldots\,$, and consequently $\varphi(3), \varphi(4), \ldots$ due to equality \eqref{eq:phi(n+1)}. In the case under consideration, the system \eqref{eq:thm_system} is
\begin{align*}
\begin{cases}
&m^{(2)}_{n}=\left(m^{(1)}_{n-2}-\sum_{i=0}^{n-1}m^{(2)}_{i}x^{(1)}_{n-i}-\sum_{i=0}^{1}m^{(2)}_i\sum_{j=0}^{1-i}x^{(1)}_j\mathbbm{1}_{\{n=2\}}\right)/x^{(1)}_0\\
&m^{(1)}_{n}=\left(m^{(2)}_{n-2}-\sum_{i=0}^{n-1}m^{(1)}_{i}x^{(2)}_{n-i}-\sum_{i=0}^{1}m^{(1)}_i\sum_{j=0}^{1-i}x^{(2)}_j\mathbbm{1}_{\{n=2\}}\right)/x^{(2)}_0
\end{cases},\,n=2,\,3,\,\ldots
 \end{align*}

Having $\varphi(1)$, $\varphi(2)$, $\varphi(3)$ and $\varphi(4)$ we use the recurrence \eqref{eq:rec} in order to find $\f(0)$
\begin{align*}
&\varphi(0)=\sum_{\substack{i_1\leqslant 1\\i_1+i_2\leqslant 3}}\mathbb{P}(X_1=i_1)\mathbb{P}(X_2=i_2)\,
\varphi(4-i_1-i_2)\\&=x^{(1)}_0x^{(2)}_0\f(4)
+(x^{(1)}_0x^{(2)}_1+x^{(1)}_1x^{(2)}_0)\f(3)
+(x^{(1)}_0x^{(2)}_2+x^{(1)}_1x^{(2)}_1)\f(2)
+(x^{(1)}_0x^{(2)}_3
+x^{(1)}_1x^{(2)}_2)\f(1).
\end{align*}
Let us recall that the same recurrence \eqref{eq:rec} can be used to calculate $\f(u)$ when $u\geqslant5$.

We provide the obtained survival probabilities in Table \ref{table1}.
\begin{table}[H]
\begin{center}
\begin{tabular}{|c||c|c|c|c|c|c|c|c|c|c|c|}
\hline
$u$&0&1&2&3&4&5&10&15\\
\hline
$\varphi(u)$&0.442&0.650&0.790&0.876&0.928&0.958&0.997&1\\
\hline
\end{tabular}
\caption{Survival probabilities for $\kappa=2$, $N=2$, $X_1\sim\mathcal{P}(1, 0)$ and $X_2\sim\mathcal{P}(2, 0)$}\label{table1}
\end{center}
\end{table}
The provided values of $\f(u)$ match the ones given in \cite[Tab. 1]{AG22}, where they were obtained by a different method.

Based on Theorem \ref{thm:gen}, we now set up the survival probability-generating function $\Xi(s)$, i. e. 
$$
\frac{1}{n!}\frac{d^n}{ds^n}\left(\Xi(s)\right)\Big|_{s=0}=\f(n+1),\,n=0,\,1,\,\ldots
$$
So, having $m^{(1)}_0$, $m^{(1)}_1$, $m^{(2)}_0$, $m^{(2)}_1$ and omitting the elementary calculation, we get

$$
\Xi(s)=\frac{(0.187+0.442s)s^2+(0.0224+0.104s)e^s}{e^{3(s-1)}-s^4},\,|s|<1,\, e^{3(s-1)}\neq s^4.
$$

\begin{ex}\label{ex:2}
Let us consider the model \eqref{model} when $\ka=2$ and $X_1\sim\mathcal{P}(1, 1)$ and $X_2\sim\mathcal{P}(9/10, 1)$. We find $\f(u)$ when $u=0,\,1,\,\ldots,\,50$ and set up the survival probability-generating function $\Xi(s)$.
\end{ex}

According to \eqref{8.1} and \eqref{8.2}, we check that the net profit condition is satisfied: $\E S_2=1+1+0.9+1=3.9<4=\kappa N$. The probability-generating function of $S_2=X_1+X_2$ is
$
G_{S_2}(s)=s^2e^{1.9(s-1)}
$
and the equation $G_{S_2}(s)=s^4$ has one non-zero solution inside the unit circle: $\alpha=-0.2928$. Since $x_0^{(1)}=0$, $x_0^{(2)}=0$ we use \eqref{eq:thm_eq_mean} and \eqref{eq:thm_zero} to set up the system
\begin{align}\label{ex:syst_2x2}
\begin{pmatrix}
x^{(2)}_1\al&\frac{G_{X_2}(\alpha)}{\alpha} x^{(1)}_1\\
x^{(2)}_1&x^{(1)}_1
\end{pmatrix}
\times
\begin{pmatrix}
m^{(1)}_0\\
m^{(2)}_0
\end{pmatrix}
=
\begin{pmatrix}
0\\
0.1
\end{pmatrix}.
\end{align}
It is easy to see that system \eqref{ex:syst_2x2} can be expressed using matrices $M_1, M_2$ and $G_2$ as in the previous example, where
\begin{align*}
M_1=\begin{pmatrix}
x^{(2)}_1\al\\
x^{(2)}_1
\end{pmatrix},\,
M_2
=\begin{pmatrix}
x^{(1)}_1\al\\
x^{(1)}_1
\end{pmatrix},\,
G_2=
\begin{pmatrix}
\frac{G_{X_2}(\alpha)}{\alpha^2}\\
1
\end{pmatrix}.
\end{align*}
The system \eqref{ex:syst_2x2} implies $m^{(1)}_0=0.1270$, $m^{(2)}_0=0.1315$ and consequently, $\varphi(1)=m^{(1)}_0=0.1270$. To proceed calculating $\varphi(u)=\sum_{i=1}^{u-1}m^{(1)}_i$, $u\geqslant 2$ we use \eqref{eq:thm_system}, which in this particular case is
\begin{align*}
\begin{cases}
&m^{(2)}_n x^{(1)}_0=m^{(1)}_{n-2}-\sum_{i=0}^{n-1}m^{(2)}_ix^{(1)}_{n-i}-\sum_{i=0}^{1}m_i^{(2)}\sum_{j=0}^{1-i}x^{(1)}_j\mathbbm{1}_{\{n=2\}}\\
&m^{(1)}_n x^{(2)}_0=m^{(2)}_{n-2}-\sum_{i=0}^{n-1}m^{(1)}_ix^{(2)}_{n-i}-\sum_{i=0}^{1}m_i^{(1)}\sum_{j=0}^{1-i}x^{(2)}_j\mathbbm{1}_{\{n=2\}}
\end{cases},\,n=2,\,3,\,\ldots,
\end{align*}
or, equivalently,
\begin{align*}
\begin{cases}
&m^{(2)}_{n-1}=\left(m^{(1)}_{n-2}-\sum_{i=0}^{n-2}m^{(2)}_ix^{(1)}_{n-i}-m_0^{(2)}x^{(1)}_1\mathbbm{1}_{\{n=2\}}\right)/x^{(1)}_1\\
&m^{(1)}_{n-1}=\left(m^{(2)}_{n-2}-\sum_{i=0}^{n-2}m^{(1)}_ix^{(2)}_{n-i}-m_0^{(1)}x^{(2)}_1\mathbbm{1}_{\{n=2\}}\right)/x^{(2)}_1
\end{cases},\,n=2,\,3,\,\ldots
\end{align*}
Substituting $n=2,\, 3,\, \ldots$ into the last two equations, we obtain $m^{(1)}_1, m^{(1)}_2, \ldots$ The survival probability $\f(0)$ is found using recurrence \eqref{eq:rec}:
\begin{align*}
\varphi(0)=\sum_{\substack{i_1\leqslant 1\\i_1+i_2\leqslant 3}}\mathbb{P}(X_1=i_1)\mathbb{P}(X_2=i_2)\,
\varphi(4-i_1-i_2)=x^{(1)}_1x^{(2)}_1\f(2)+x^{(1)}_1x^{(2)}_2\f(1).
\end{align*}
    After completing all the necessary calculations, we get survival probabilities which are provided in Table \ref{table2}.
\begin{table}[H]
\centering
\begin{tabular}{|c||c|c|c|c|c|c|c|c|c|c|c|}
\hline
$u$&0&1&2&3&4&5&10&20&30&40&50  \\
\hline             $\varphi(u)$&0.048&0.127&0.209&0.286&0.355&0.417&0.649&0.873&0.954&0.983&0.994\\
\hline
\end{tabular}
\caption{Survival probabilities for $\kappa=2$, $N=2$, $X_1\sim\mathcal{P}(1, 1)$ and $X_2\sim\mathcal{P}(9/10, 1)$}\label{table2}
\end{table}
    Once again, obtained results in Table \ref{table2} match the ones presented in \cite[Tab. 3]{AG22}, where the numbers are obtained differently, i. e. calculating limits of certain recurrent sequences.
	
Theorem \ref{thm:gen} yields the following survival probability-generating function

$$
\Xi(s)=\frac{0.0516e^{s-1}+0.0484s}{e^{1.9(s-1)}-s^2},\,|s|<1,\,e^{1.9(s-1)}\neq s^2.
$$

\begin{ex}\label{ex:3}
Let us consider the bi-seasonal model \eqref{model} with $\ka=3$ where claims are represented by two independent random variables $X_1$ and $X_2$, whose distributions are given in Table \ref{table3} and Table\ref{table4}.
\begin{table}[H]
\centering
\begin{tabular}{|c||c|c|c|c|c|}
\hline
$i$&0&1&2&3&4  \\
\hline
$\mathbb{P}(X_1=i)$&0.4096&0.4096&0.1536&0.0256&0.0016\\
\hline
\end{tabular}
\caption{Probability distribution of random variable $X_1$} \label{table3}
\end{table}
\begin{table}[H]
\centering
\begin{tabular}{|c||c|c|c|c|c|}
\hline
$i$&0&1&2  \\
\hline
$\mathbb{P}(X_2=i)$&0.04&0.32&0.64\\
\hline
\end{tabular}
\caption{Probability distribution of random variable $X_2$}\label{table4}
\end{table}
We find the survival probability $\f(u)$ for all $u\in\mathbb{N}_0$ and its generating function $\Xi(s)$.
\end{ex}

It is easy to observe that $\E S_2=2.4<6$. Thus, the net profit condition is valid. Sum's $X_1+X_2$ probability-generating function is
\begin{align*}
G_{S_2}(s)=\left(0.4096+0.4096s+0.1536s^2+0.0256s^3+0.0016s^4\right)\left(0.04+0.32s+0.64s^2\right).
\end{align*}

Solving $G_{S_2}(s)=s^6$, we obtain the following roots inside the unit circle:
$$
\alpha_1=-\frac{4}{11}, \,\alpha_2=-0.2250, \,\alpha_3=-0.0154-0.7423i, \,\alpha_4=-0.0154+0.7423i.
$$
Note that the complex roots always occur in conjugate pairs due to $G_{S_N}(\overline{s})-\overline{s}^{\ka N}=\overline{G_{S_N}(s)-s^{\ka N}}$, where over-line denotes conjugate.
According to Lemma \ref{lem:roots}, there must be one root of multiplicity two and one may check that $\alpha_1$ is such. 

We then employ \eqref{eq:thm_derivative} to create the modified versions of $M_1$, $M_2$ and $G_2$. Let $\tilde{M}_1$, $\tilde{M}_2$ and $\tilde{G}_2$ be
\begin{align*}
\tilde{M}_1:=
\begin{pmatrix}
\al_1^2F_{X_2}(2)+\al_1F_{X_2}(1)+x_0^{(2)}&\al_1^2F_{X_2}(1)+\al_1x_0^{(2)}&x_0^{(2)}\alpha_1^2\\
\al_2^2F_{X_2}(2)+\al_2F_{X_2}(1)+x_0^{(2)}&\al_2^2F_{X_2}(1)+\al_2x_0^{(2)}&x_0^{(2)}\alpha_2^2\\
\al_3^2F_{X_2}(2)+\al_3F_{X_2}(1)+x_0^{(2)}&\al_3^2F_{X_2}(1)+\al_3x_0^{(2)}&x_0^{(2)}\alpha_3^2\\
\al_4^2F_{X_2}(2)+\al_4F_{X_2}(1)+x_0^{(2)}&\al_4^2F_{X_2}(1)+\al_4x_0^{(2)}&x_0^{(2)}\alpha_4^2\\
F_{X_2}(1)+2\al_1F_{X_2}(2)&x_0^{(2)}+2\al_1F_{X_2}(1)&2x^{(2)}_0\al_1\\
x_2^{(2)}+2x_1^{(2)}+3x_0^{(2)}&x_1^{(2)}+2x_0^{(2)}&x_0^{(2)}
\end{pmatrix},
\end{align*}

\begin{align*}
\tilde{M}_2:=
\begin{pmatrix}
\al_1^2F_{X_1}(2)+\al_1F_{X_1}(1)+x_0^{(1)}&\al_1^2F_{X_1}(1)+\al_1x_0^{(1)}&x_0^{(1)}\alpha_1^2\\
\al_2^2F_{X_1}(2)+\al_2F_{X_1}(1)+x_0^{(1)}&\al_2^2F_{X_1}(1)+\al_2x_0^{(1)}&x_0^{(1)}\alpha_2^2\\
\al_3^2F_{X_1}(2)+\al_3F_{X_1}(1)+x_0^{(1)}&\al_3^2F_{X_1}(1)+\al_3x_0^{(1)}&x_0^{(1)}\alpha_3^2\\
\al_4^2F_{X_1}(2)+\al_4F_{X_1}(1)+x_0^{(2)}&\al_4^2F_{X_1}(1)+\al_4x_0^{(1)}&x_0^{(1)}\alpha_4^2\\
\tilde{M}_{5,\,1}&\tilde{M}_{5,\,2}&\tilde{M}_{5,\,3}\\
3x_0^{(1)}+2x_1^{(1)}+x_2^{(1)}&2x_0^{(1)}+x_1^{(1)}&x_0^{(1)}
\end{pmatrix},
\end{align*}

$$
\begin{pmatrix}
\tilde{M}_{5,\,1}&\tilde{M}_{5,\,2}&\tilde{M}_{5,\,3}
\end{pmatrix}
=
\begin{pmatrix}
\left(\frac{G_{X_2}(s)}{s^{2}}\right)'\Big|_{s=\al_1}
&\left(\frac{G_{X_2}(s)}{s}\right)'\Big|_{s=\al_1}
&G'_{X_2}(s)\Big|_{s=\al_1}
\end{pmatrix}
\begin{pmatrix}
x_0^{(1)}&0&0\\
F_{X_1}(1)&x_0^{(1)}&0\\
F_{X_1}(2)&F_{X_1}(1)&x_0^{(1)}
\end{pmatrix},
$$

\begin{align*}
\tilde{G}_2:=
\begin{pmatrix}
{G_{X_2}(\al_1)}/{\al_1^3}&{G_{X_2}(\al_1)}/{\al_1^3}&{G_{X_2}(\al_1)}/{\al_1^3}\\
{G_{X_2}(\al_2)}/{\al_2^3}&{G_{X_2}(\al_2)}/{\al_2^3}&{G_{X_2}(\al_2)}/{\al_2^3}\\
{G_{X_2}(\al_3)}/{\al_3^3}&{G_{X_2}(\al_3)}/{\al_3^3}&{G_{X_2}(\al_3)}/{\al_3^3}\\
{G_{X_2}(\al_4)}/{\al_4^3}&{G_{X_2}(\al_4)}/{\al_4^3}&{G_{X_2}(\al_4)}/{\al_4^3}\\
1&1&1\\
1&1&1
\end{pmatrix}.
\end{align*}

Then,
        \begin{equation*}
\begin{pmatrix}
\tilde{M}_1&\tilde{M}_2\circ \tilde{G}_2
\end{pmatrix}_{6\times6}
\begin{pmatrix}
m^{(1)}_0\\m^{(1)}_1\\m^{(1)}_2\\
m^{(2)}_0\\m^{(2)}_1\\m^{(2)}_2
\end{pmatrix}
=
\begin{pmatrix}
0\\0\\0\\0\\0\\3.6
\end{pmatrix}
\quad \Rightarrow \quad
\begin{pmatrix}
    			m^{(1)}_0\\m^{(1)}_1\\m^{(1)}_2\\
    			m^{(2)}_0\\m^{(2)}_1\\m^{(2)}_2
    		\end{pmatrix}
    		=
    		\begin{pmatrix}
    			0.9984\\0.0016\\0\\
    			1\\0\\0
    		\end{pmatrix}
\end{equation*}

        It follows that $\varphi(1)=m_0^{(1)}=0.9984$, $\varphi(2)=m_0^{(1)}+m_1^{(1)}=1$ and consequently $\vphi(u)=1$ for all $u\geqslant3$. Consequently, 
        
\begin{align*}
\varphi(0)&=\sum_{\substack{i_1\leqslant 2\\i_1+i_2\leqslant 5}}\mathbb{P}(X_1=i_1)\mathbb{P}(X_2=i_2)\,\varphi(6-i_1-i_2)\\
&
=x^{(1)}_0x^{(2)}_0\f(6)
+\left(x^{(1)}_0x^{(2)}_1+x^{(1)}_1x^{(2)}_0\right)\f(5)
+\left(x^{(1)}_0x^{(2)}_2
+x^{(1)}_1x^{(2)}_1+x^{(1)}_2x^{(2)}_0\right)\f(4)
\\
&\quad
+\left(x^{(1)}_2x^{(2)}_1+x^{(1)}_1x^{(2)}_2\right)\f(3)
+x^{(1)}_2x^{(2)}_2\f(2)
=0.9728.
\end{align*}
        
The correctness of these results can be verified in the following way. If initial surplus $u=1$, ruin can only occur at the first moment of time and only if $1+3\cdot1-X_1\leqslant0$, i. e. $X_1=4$. Thus, $\f(1)=1-\T(X_1=4)=1-0.0016=0.9984$. If initial surplus $u\geqslant 1$, then ruin will never occur. There are two reasons for that. First of all, in the first moment of time insurer's wealth will never drop below one. Moreover, every two periods insurer earns $6$ units of currency and that is the maximum amount of claims that the insurer can suffer during two consecutive periods. The result of $\f(0)$ is also logical as with no initial capital ruin can occur only if $X_1=3$ or $X_1=4$, thus $\vphi(0)=1-\T(X_1=3)-\T(X_1=4)=1-0.0256-0.0016=0.9728$.

The generating function of $\f(1),\,\f(2),\,\ldots$ in the considered case is simple
$$
\Xi(s)=0.9984+s+s^2+s^3+\ldots=\frac{1}{1-s}-0.0016,\,|s|<1.
$$
One may verify that Theorem \ref{thm:gen} produces the same result.

\begin{ex}\label{ex:4}
In the last example, we consider ten season model with a premium rate of $5$, i. e. $N=10$, $\kappa=5$, and we assume claims to be generated by independent random variables $X_k\sim\mathcal{P}(k/(k+1)+4,\,0)$, $k\in\{1,\, 2,\,\ldots,\, 10\}$, where $\mathcal{P}(\lambda,\,0)$ denotes Poisson distribution with parameter $\lambda$. We calculate both the finite time survival probability $\f(u, T)$ and the ultimate time survival probability $\f(u)$ and provide a frame of ultimate time survival probability-generating function $\Xi(s)$.
\end{ex}

Let us verify that the net profit condition is satisfied:
\begin{equation*}
\E S_{10}=\sum_{i=1}^{10}\E X_k=\sum_{k=1}^{10}\left(\frac{k}{k+1}+4\right)=\frac{1330009}{27720}\approx47.9801<50.
\end{equation*}

We now apply Theorem \ref{thm:1}. The equation
\begin{equation*}
G_{S_{10}}(s)=e^{1330009(s-1)/27720}=s^{50}
\end{equation*}
has $49$ simple roots inside the unit circle depicted in Figure \ref{fig:roots}.

\begin{figure}[H]
\begin{center}
\includegraphics[scale=0.75]{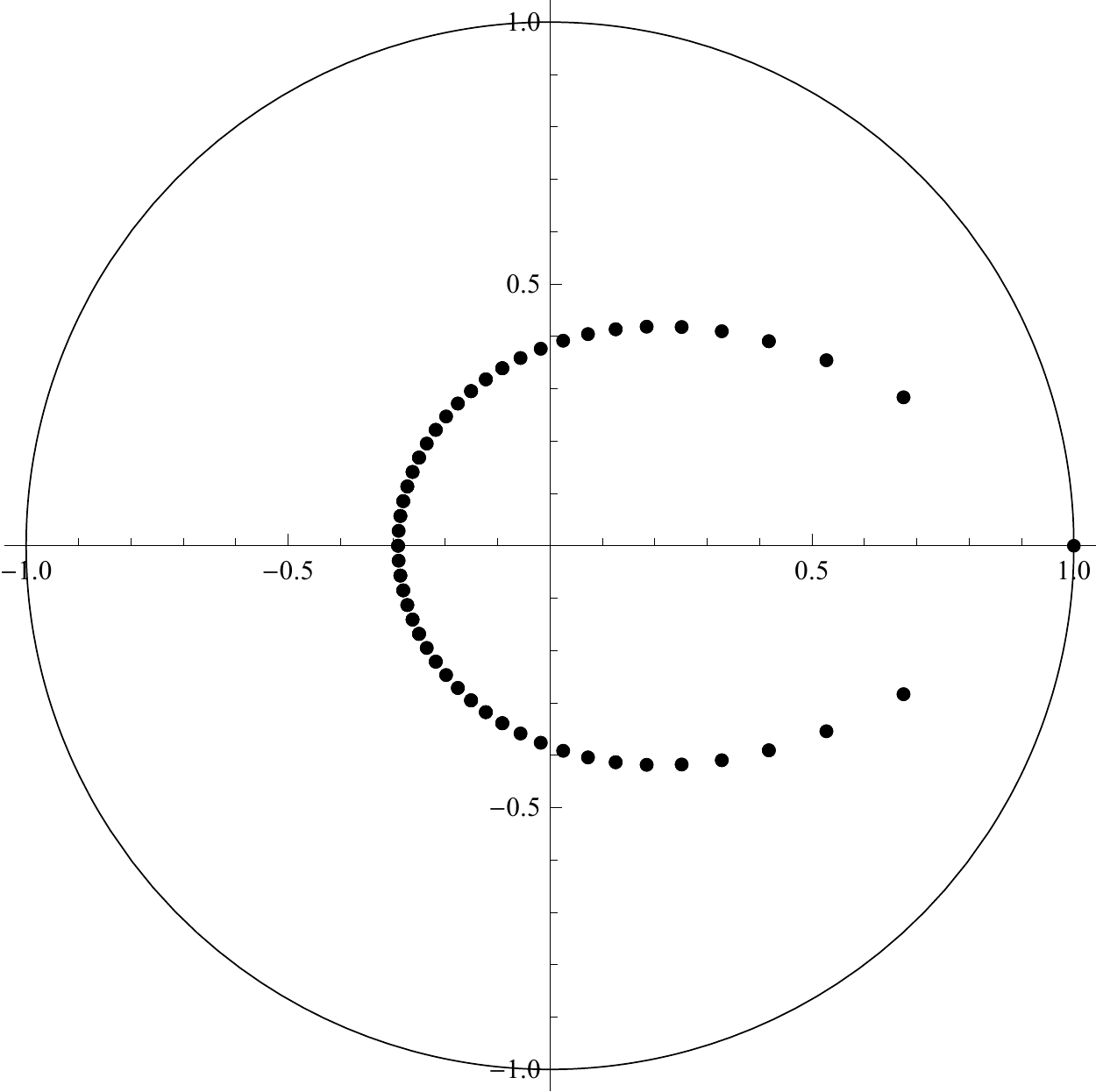}
\caption{Roots of $s^{50}=G_{S_{10}}(s)$, when $X_k\sim\mathcal{P}(k/(k+1)+4,\,0)$, $k\in\{1,\, 2,\,\ldots,\, 10\}$.}\label{fig:roots}
\end{center}
\end{figure}

Denoting these roots by $\al_1,\, \al_2,\, \ldots,\, \al_{49}$
we set up matrices $M_1$, $M_2$, $\ldots\,$, $M_{10}$ and $G_2$, $G_3$, $\ldots\,$, $G_{10}$:
\begin{align*}
&M_1=
\begin{pmatrix}
\sum\limits_{j=0}^{4}\al_1^jF_{X_{10}}(j) & \sum\limits_{j=1}^{4}\al_1^jF_{X_{10}}(j-1)&\ldots&x^{(10)}_0\al^4_1\\
\vdots&\vdots&\ddots&\vdots\\
\sum\limits_{j=0}^{4}\al_{49}^jF_{X_{10}}(j) & \sum\limits_{j=1}^{4}\al_{49}^jF_{X_{10}}(j-1)&\ldots&x^{(10)}_0\al^4_{49}\\
\sum\limits_{j=0}^{4}x^{(10)}_j(5-j) & \sum\limits_{j=0}^{3}x^{(10)}_j(5-j-1) & \ldots&x^{(10)}_0
\end{pmatrix},
\end{align*}
\begin{align*}
&M_2=
\begin{pmatrix}
\sum\limits_{j=0}^{4}\al_1^jF_{X_{1}}(j) & \sum\limits_{j=1}^{4}\al_1^jF_{X_{1}}(j-1)&\ldots&x^{(1)}_0\al^4_1\\
\vdots&\vdots&\ddots&\vdots\\
\sum\limits_{j=0}^{4}\al_{49}^jF_{X_{1}}(j) & \sum\limits_{j=1}^{4}\al_{49}^jF_{X_{1}}(j-1)&\ldots&x^{(1)}_0\al^4_{49}\\
\sum\limits_{j=0}^{4}x^{(1)}_j(5-j) & \sum\limits_{j=0}^{3}x^{(1)}_j(5-j-1) & \ldots&x^{(1)}_0
\end{pmatrix},\,\ldots
\end{align*}

\begin{align*}
 &M_{10}=
\begin{pmatrix}
\sum\limits_{j=0}^{4}\al_1^jF_{X_{9}}(j) & \sum\limits_{j=1}^{4}\al_1^jF_{X_{9}}(j-1)&\ldots&x^{(9)}_0\al^4_1\\
\vdots&\vdots&\ddots&\vdots\\
\sum\limits_{j=0}^{4}\al_{49}^jF_{X_{9}}(j) & \sum\limits_{j=1}^{4}\al_{49}^jF_{X_{9}}(j-1)&\ldots&x^{(9)}_0\al^4_{49}\\
\sum\limits_{j=0}^{4}x^{9}_j(5-j) & \sum\limits_{j=0}^{3}x^{9}_j(5-j-1) & \ldots&x^{9}_0
\end{pmatrix},
\end{align*}	
\begin{align*}
G_2=
\begin{pmatrix}
\frac{G_{X_{10}}(\al_1)}{\al^5_1}
&\ldots& \frac{G_{X_{10}}(\al_1)}{\al^5_1}\\
\vdots&\ddots&\vdots\\
\frac{G_{X_{10}}(\al_{49})}{\al^5_{49}}
& \ldots &\frac{G_{X_{10}}(\al_{49})}{\al^5_{49}}\\
1 &\ldots& 1
\end{pmatrix},\,
G_3=
\begin{pmatrix}
\frac{G_{X_{10}+X_1}(\al_1)}{\al^{10}_1}
&\ldots& \frac{G_{X_{10}+X_1}(\al_1)}{\al^{10}_1}\\
\vdots&\ddots&\vdots\\
\frac{G_{X_{10}+X_1}(\al_{49})}{\al^{10}_{49}}
& \ldots &\frac{G_{X_{10}+X_1}(\al_{49})}{\al^{10}_{49}}\\
1 &\ldots& 1
\end{pmatrix},\,\ldots
\end{align*}
\begin{align*}
G_{10}=
\begin{pmatrix}
\frac{G_{X_{10}+X_1+\ldots+X_8}(\al_1)}{\al^{45}_1}
&\ldots& \frac{G_{X_{10}+X_1+\ldots+X_8}(\al_1)}{\al^{45}_1}\\
\vdots&\ddots&\vdots\\
\frac{G_{X_{10}+X_1+\ldots+X_8}(\al_{49})}{\al^{45}_{49}}
& \ldots &			\frac{G_{X_{10}+X_1+\ldots+X_8}(\al_{49})}{\al^{45}_{49}}\\
1 &\ldots& 1
\end{pmatrix}.
\end{align*}	
Solving the system
\begin{align}\label{syst:ex_50}
\begin{pmatrix}
M_1 & M_2\circ G_2 & \ldots & M_{10}\circ G_{10}
\end{pmatrix}_{50\times50}
\begin{pmatrix}
m^{(1)}_0\\
m^{(1)}_1\\
\vdots\\
m^{(1)}_4\\
\vdots\\
m^{(10)}_0\\
m^{(10)}_1\\
\vdots\\
m^{(10)}_4\\			
\end{pmatrix}_{50\times1}
=
\begin{pmatrix}
0\\
0\\
\vdots\\
0\\
\frac{55991}{27720}
\end{pmatrix}_{50\times1}		
\end{align}
we obtain $m^{(1)}_0=0.1821$, $m^{(1)}_1=0.0604$, $m^{(1)}_2=0.0583$, $m^{(1)}_3=0.0545$, $m^{(1)}_4=0.0504$. 

Therefore, using \eqref{eq:phi(n+1)}:
\begin{align*}
\f(1)&=m^{(1)}_0=0.1821,\\ 
\f(2)&=m^{(1)}_0+m^{(1)}_1=0.2425,\\
\f(3)&=m^{(1)}_0+m^{(1)}_1+m^{(1)}_2=0.3009,\\
\f(4)&=m^{(1)}_0+m^{(1)}_1+m^{(1)}_2+m^{(1)}_3=0.3554,\\
\f(5)&=m^{(1)}_0+m^{(1)}_1+m^{(1)}_2+m^{(1)}_3+m^{(1)}_4=0.4058.
\end{align*}

Employing system \eqref{eq:thm_system} we find the remaining probabilities $m^{(1)}_5, m^{(1)}_6, \ldots$:
\begin{align*}
\begin{cases}
&m^{(2)}_{n}=\left(m^{(1)}_{n-5}-\sum_{i=0}^{n-1}m^{(2)}_{i}x^{(1)}_{n-i}-\sum_{i=0}^{4}m^{(2)}_i\sum_{j=0}^{4-i}x^{(1)}_j\mathbbm{1}_{\{n=5\}}\right)/x^{(1)}_0\\
&\hspace{0,9cm}\vdots\\
&m^{(10)}_{n}=\left(m^{(9)}_{n-5}-\sum_{i=0}^{n-1}m^{(10)}_{i}x^{(9)}_{n-i}-\sum_{i=0}^{4}m^{(10)}_i\sum_{j=0}^{4-i}x^{(9)}_j\mathbbm{1}_{\{n=5\}}\right)/x^{(9)}_0\\
&m^{(1)}_{n}=\left(m^{(10)}_{n-5}-\sum_{i=0}^{n-1}m^{(1)}_{i}x^{(10)}_{n-i}-\sum_{i=0}^{4}m^{(1)}_i\sum_{j=0}^{4-i}x^{(10)}_j\mathbbm{1}_{\{n=5\}}\right)/x^{(10)}_0
\end{cases},\,n=5,\, 6,\,\ldots\,
\end{align*}
We substitute the obtained probabilities $m^{(1)}_0,\,m^{(1)}_1,\,\ldots$ into \eqref{eq:phi(n+1)} and calculate $\varphi(6), \varphi(7), \ldots$
Finally, $\f(0)$ can be found using \eqref{eq:rec}:
\begin{equation*}
\varphi(0)=\sum_{\substack{i_1\leqslant 4\\i_1+i_2\leqslant 9\\
i_1+i_2+i_3\leqslant 14\\ \vdots\\i_1+i_2+\ldots+i_{10}\leqslant 49}}\mathbb{P}(X_1=i_1)\mathbb{P}(X_2=i_2)\cdots\mathbb{P}(X_{10}=i_{10})
\varphi\left(50-\sum_{j=1}^{10}i_j\right).
\end{equation*}

The final results, including the finite time ruin probability calculated by Theorem \ref{thm:fin_t}, rounded up to three decimal places, are provided in Table \ref{table5}.


\begin{table}[H]
\centering
\begin{tabular}{|c|c|c|c|c|c|c|c|c|c|}
\hline
{$T$}&{$u=0$}&{$u=1$}&{$u=2$}&{$u=3$}&{$u=4$}&{$u=5$}&{$u=10$}&{$u=20$}&{$u=30$}\\
\hline
$1$&$0.532$&$0.703$&$0.831$&$0.913$&$0.960$&$0.983$&$1$&$1$&$1$\\
$2$&$0.424$&$0.587$&$0.727$&$0.831$&$0.902$&$0,946$&$0.999$&$1$&$1$\\
$3$&$0.368$&$0.520$&$0.657$&$0.767$&$0.849$&$0.906$&$0.995$&$1$&$1$\\
$4$&$0.332$&$0.474$&$0.606$&$0.717$&$0.804$&$0.869$&$0.988$&$1$&$1$\\
$5$&$0.306$&$0.440$&$0.567$&$0.677$&$0.766$&$0.834$&$0.979$&$1$&$1$\\
$10$&$0.235$&$0.343$&$0.450$&$0.548$&$0.635$&$0.708$&$0.921$&$0.998$&$1$\\
$20$&$0.200$&$0.294$&$0.389$&$0.478$&$0.558$&$0.629$&$0.863$&$0.990$&$1$\\
$30$&$0.179$&$0.264$&$0.350$&$0.432$&$0.507$&$0.575$&$0.814$&$0.979$&$0.999$\\
\hline
$\infty$&$0.125$&$0.182$&$0.243$&$0.301$&$0.355$&$0.406$&$0.605$&$0.826$&$0.923$\\
\hline
\end{tabular}
\caption{Survival probabilities for $\kappa=5$, $N=10$, $X_k\sim\mathcal{P}(k/(k+1)+4,\,0)$, $k\in\{1, 2,\ldots, 10\}$}\label{table5}
\end{table}

The survival probability $\f(1),\,\f(2),\,\ldots$ generating function, having $m_i^{(j)}$,  $i=0,\,1,\,\ldots,\,4$, $j=1,\,2,\,\ldots,\,10$ from system \eqref{syst:ex_50}, when $|s|<1$ and $e^{a_{10}(s-1)}\neq s^{50}$  is

$$
\Xi(s)=\frac{\mathbbm{u}^T\mathbbm{v}}{e^{a_{10}(s-1)}-s^{50}},\,\mathbbm{u}=
\begin{pmatrix}
s^{45}\\
s^{40}e^{a_1(s-1)}\\
s^{35}e^{a_2(s-1)}\\
\vdots\\
s^5e^{a_8(s-1)}\\
e^{a_9(s-1)}
\end{pmatrix},\,
\mathbbm{v}=
\begin{pmatrix}
e^{-\la_1}\sum_{i=0}^{4}m_i^{(2)}\sum_{j=i}^{4}s^j\sum_{l=0}^{j-i}\la_1^j/l!\\
e^{-\la_2}\sum_{i=0}^{4}m_i^{(3)}\sum_{j=i}^{4}s^j\sum_{l=0}^{j-i}\la_2^j/l!\\
\vdots\\
e^{-\la_9}\sum_{i=0}^{4}m_i^{(10)}\sum_{j=i}^{4}s^j\sum_{l=0}^{j-i}\la_9^j/l!\\
e^{-\la_{10}}\sum_{i=0}^{4}m_i^{(1)}\sum_{j=i}^{4}s^j\sum_{l=0}^{j-i}\la_{10}^j/l!
\end{pmatrix},
$$
where $a_n=4n+\sum_{k=0}^{n}k/(k+1)$ and $\la_n=4+n/(n+1)$ when $n=1,\,2,\,\ldots,\,10$.

\typeout{}
\bibliography{mybibfile}

\end{document}